\documentclass[reqno, eucal]{amsart}

\usepackage{epic,eepic}
\usepackage[mathscr]{eucal} %  use \EuScript (\mathcal unchanged)
\usepackage{psfrag}     %<======added 

\usepackage{amsmath}
\usepackage{amsthm}

\usepackage[dvips]{graphicx,color}  %<==== combined

\usepackage{amssymb}
\usepackage{epic,eepic}
\usepackage{amscd}
\usepackage{longtable}
\usepackage{array}

\newtheorem{theorem}{Theorem}
\newtheorem{lemma}[theorem]{Lemma}

\newtheorem{proposition}[theorem]{Proposition}

\theoremstyle{definition}

\newtheorem{example}[theorem]{Example}
\newtheorem{remark}[theorem]{Remark}
\newcommand{\ot}{\otimes}

\newcommand{\Z}{{\mathbb Z}}
\newcommand{\R}{{\mathbb R}}
\newcommand{\C}{{\mathbb C}}

\begin{document}

\title[Stochastic $R$ matrix]
{Stochastic $\boldsymbol{R}$ matrix for $\boldsymbol{U_q(A^{(1)}_n)}$}
\author{A.~Kuniba}
\email{atsuo@gokutan.c.u-tokyo.ac.jp}
\address{Institute of Physics, University of Tokyo, Komaba, Tokyo 153-8902, Japan}

\author{V.~V.~Mangazeev}
\email{vladimir.mangazeev@anu.edu.au}
\address{Department of Theoretical Physics, Research School of Physics and Engineering,
Australian National University, Canberra, ACT 0200, Australia.}

\author{S.~Maruyama}
\email{maruyama@gokutan.c.u-tokyo.ac.jp}
\address{Institute of Physics, University of Tokyo, Komaba, Tokyo 153-8902, Japan}

\author{M.~Okado}
\email{okado@sci.osaka-cu.ac.jp}
\address{Department of Mathematics, Osaka City University, 
3-3-138, Sugimoto, Sumiyoshi-ku, Osaka, 558-8585, Japan}

%\date{\today}

\maketitle

\vspace{0.3cm}
\begin{center}{\bf Abstract}\end{center}
We show that the quantum $R$ matrix for symmetric tensor 
representations of $U_q(A^{(1)}_n)$ satisfies 
the sum rule required for its stochastic interpretation 
under a suitable gauge.
Its matrix elements at a special point of the 
spectral parameter 
are found to factorize into the form that naturally 
extends Povolotsky's local transition rate in the $q$-Hahn process for $n=1$.
Based on these results 
we formulate new discrete and continuous time integrable 
Markov processes on a  
one-dimensional chain in terms of $n$ species of particles obeying 
asymmetric stochastic dynamics.
Bethe ansatz eigenvalues of the Markov matrices 
are also given.

\vspace{0.4cm}

\section{Introduction}\label{sec:int}

Quantum groups and theory of quantum integrable systems 
provide efficient algebraic and analytic tools to evaluate 
non-equilibrium characteristics in stochastic processes in statistical mechanics.
See for example \cite{GS,SW,LM,P,TW,CRM, BSc,CP,BP} and references therein.
Typically in such an approach, 
one sets up a row transfer matrix or its derivative as in usual vertex 
or spin chain models \cite{Bax}, 
and seeks the situation that admits an interpretation as 
a Markov matrix of a certain dynamical system on a one-dimensional chain.
It leads to a postulate more stringent than 
the models in the equilibrium setting.
Namely, the transfer matrix 
or its derivative must have non-negative off-diagonal elements 
and they should further satisfy 
a certain sum-to-unity or sum-to-zero conditions assuring 
the total probability conservation depending on whether 
the time evolution is discrete or continuous, respectively.

One may try to modify a given transfer matrix so as to fit them,
but doing so indiscreetly leads to a loss of the essential merit, 
the {\em integrability} or put more practically, the Bethe ansatz solvability.
In this way an important general question arises;
Can one architect the transfer matrices or their constituent quantum $R$ matrices
so as to fulfill the basic axioms of Markov matrices 
without spoiling the integrability?

The aim of this paper is to answer it affirmatively for 
the $R$ matrix associated with the symmetric tensor 
representations of the Drinfeld-Jimbo quantum affine algebra $U_q(A^{(1)}_n)$ \cite{D,J}.
By now, the quantum $R$ matrix itself is a well-known {\em classic}.
Nevertheless investigation of the above question elucidates a number of 
remarkable insights which have hitherto escaped notice.

For a quick exposition, 
let $R(z)$ be the quantum $R$ matrix on 
the symmetric tensor representation $V_l \otimes V_m$ of degrees $l$ and $m$
with spectral parameter $z$.
Then there is a suitable (stochastic) gauge $S(z)$ of $R(z)$
that satisfies the sum-to-unity condition 
$\sum_{\gamma,\delta}S(z)_{\alpha, \beta}^{\gamma,\delta} =1$
(Theorem \ref{pr:s=1})  preserving the Yang-Baxter equation
(Proposition \ref{pr:S}).
Moreover its nonzero elements at $z=q^{l-m}$ are described explicitly for $l \le m$ as
$S(z=q^{l-m})_{\alpha, \beta}^{\gamma,\delta} 
= \Phi_{q^2}(\gamma | \beta; q^{-2l}, q^{-2m})$ 
$(\alpha, \beta, \gamma,\delta \in \Z_{\ge 0}^n)$
in terms of the function $\Phi_q(\gamma| \beta; \lambda,\mu)$ defined in (\ref{Pdef}) as
\begin{align*} 
q^{\sum_{1\le i<j\le n}(\beta_i-\gamma_i)\gamma_j}
\left(\frac{\mu}{\lambda}\right)^{\gamma_1+\cdots + \gamma_n}
\frac{(\lambda;q)_{\gamma_1+\cdots + \gamma_n}
({\textstyle{\frac{\mu}{\lambda}}};q)_{\beta_1+\cdots +\beta_n-\gamma_1-\cdots -\gamma_n}}
{(\mu; q)_{\beta_1+\cdots +\beta_n}}
\prod_{i=1}^n
\frac{(q;q)_{\beta_i}}
{(q;q)_{\gamma_i}(q;q)_{\beta_i-\gamma_i}}.
\end{align*} 
See Proposition \ref{pr:sp}\footnote{
For simplicity, it is quoted omitting the distinction between 
$\beta$ and $\bar{\beta}$ etc. 
}.
For $n=1$ this function emerged essentially 
in the explicit formulas of 
the $R$ matrix and the $Q$ operators for $U_q(A^{(1)}_1)$ \cite{M1,M2}.
Around the same time it was also introduced in the form
$\varphi(m|m') = \Phi_q(m|m';{\textstyle \frac{\nu}{\mu}}, \nu)|_{n=1}$
to the realm of stochastic models by Povolotsky \cite[eq.(8)]{P} motivated by \cite{EMZ}, 
which triggered many subsequent studies, e.g.  \cite{CP,BP}.

In this paper we establish the above formula for general 
$n$ substantially in Theorem \ref{pr:Rs}.
Our strategy is to resort to the {\em characterization} of 
the $R$ matrix as the commutant of $U_q(A^{(1)}_n)$ \cite{D,J},
which takes advantage of the most essential machinery of the theory rather than 
manipulating concrete formulas as in the preceding works.
Our proof of Theorem \ref{pr:s=1} also captures the sum-to-unity relations 
(\ref{msk}) conceptually 
from the representation theory of quantum groups.
It manifests that the totality of those relations is nothing but 
the {\em $U_q(A_n)$-orbit} of the unit normalization condition (\ref{nor}) 
on the trivial highest weight vector.
Such a mechanism is quite likely to 
work similarly in many other algebras and representations.

Based on these findings on the $R$ matrices, 
we first formulate two kinds of  
commuting families of discrete time Markov processes on a one-dimensional chain.
They are described in terms of $n$ species of particles obeying 
totally asymmetric dynamics 
with and without constraint on their numbers 
occupying a site or hopping to the right at one time step.  
From the constraint-free case we then 
further extract the continuous time versions 
by differentiating the Markov transfer matrix by $\lambda$ 
in $\Phi_q(\gamma | \beta; \lambda, \mu)$ parameterizing the
commuting family.
The procedure is analogous to the standard derivation of 
spin chain Hamiltonians as in \cite[eq.(10.14.20)]{Bax}.
A curiosity encountered in our model is that the transfer matrix admits
{\em two} ``Hamiltonian points"  
$\lambda=1$ and $\lambda=\mu$ at which such calculations can naturally be executed
as in (\ref{bax}).
They lead to the two Markov matrices $H$ and $\hat{H}$ which are
interpreted as $n$-species {\em totally asymmetric zero range processes} (TAZRPs)
in which particles hop to the right and to the left adjacent site, respectively. 
By the construction the commutativity  
$[H, \hat{H}]=0$ holds, therefore the superposition
$a H + b\hat{H}$ yields an integrable asymmetric 
zero range process in which $n$ species of particles can hop to {\em either} direction.

In the TAZRP corresponding to $H$,  the local transition rate is given by
\begin{align*}
q^{\sum_{1\le i<j \le n}(\beta_i-\gamma_i)\gamma_j}
\frac{\mu^{\gamma_1+\cdots + \gamma_n-1}(q;q)_{\gamma_1+\cdots + \gamma_n-1}}
{(\mu q^{\beta_1+\cdots +\beta_n-\gamma_1-\cdots -\gamma_n};
q)_{\gamma_1+\cdots + \gamma_n}}
\prod_{i=1}^n
\frac{(q;q)_{\beta_i}}
{(q;q)_{\gamma_i}(q;q)_{\beta_i-\gamma_i}}
\end{align*}
for the nontrivial process\footnote{This is (\ref{rate}) with $\epsilon = 1$.
``Nontrivial" means $\gamma_1+ \cdots + \gamma_n \ge 1$.
In general the rate is given by 
$-\epsilon \mu^{-1} \times (\ref{yum})|_{\alpha\rightarrow \beta}$.} in which 
$\gamma_i$ among the $\beta_i$ particles of species $i$
in the departure site are moving out $(\gamma_i \le \beta_i)$. 
When $\mu=0$, the transitions are limited to the case 
$\gamma_1+\cdots + \gamma_n =1$, and the model reduces to 
the $n$-species $q$-boson process 
derived in \cite{T} whose $n=1$ case further goes back to \cite{SW}.
When $n=1$, the above transition rate for general $\mu$ reproduces the one in 
\cite[p2]{T0} by a suitable adjustment.

In the TAZRP associated to $\hat{H}$, the relevant transition rate 
(\ref{rate2}) is similar to the above.
In particular, at $\mu=0$ and $\epsilon=1$ it reduces to
\begin{align*}
q^{\sum_{1 \le i<j \le n}\gamma_i(\beta_j-\gamma_j)}
(q;q)_{\gamma_1+\cdots + \gamma_n-1}
\prod_{i=1}^n
\frac{(q;q)_{\beta_i}}
{(q;q)_{\gamma_i}(q;q)_{\beta_i-\gamma_i}}.
\end{align*}
At $q= 0$, it gives 
rise to a kinematic constraint $\sum_{1 \le i<j \le n}\gamma_i(\beta_j-\gamma_j)=0$
which is translated into a simple priority rule on the species of particles that are 
jumping out together.
It precisely reproduces the $n$-species TAZRP  
explored in \cite{KMO, KMO2} under a suitable adjustment of conventions.

Once the models are identified in the framework  
of quantum integrable systems,
spectra of the Markov matrices with the 
periodic boundary condition follow from the Bethe ansatz.
We present the eigenvalue formulas adjusted to 
the stochastic setting under consideration.  
Steady state eigenvalues, given explicitly in (\ref{esst}), are naturally identified 
with those associated with the trivial Baxter $Q$ functions.

The layout of the paper is as follows.
In Section \ref{sec:R} we derive several 
properties of the $U_q(A^{(1)}_n)$ quantum $R$ matrix $R(z)$ 
and its stochastic versions $S(z)$ and $\mathscr{S}(\lambda,\mu)$
that are essential for applications in the subsequent sections.
In Section \ref{sec:sm}  
the commuting transfer matrices built upon the 
$S(z)$ and $\mathscr{S}(\lambda,\mu)$ are shown to satisfy the 
basic axioms of Markov matrices in a certain range of parameters.
The associated stochastic processes are formulated,
which generalize various known models for $n=1$.
Section \ref{sec:ba} presents the 
Bethe ansatz eigenvalue formulas of the Markov matrices
together with some examples of steady states.
Section \ref{sec:end} is a summary.
Appendix \ref{app:R} contains explicit forms of 
simple examples of the $R$ matrix.

Throughout the paper we fix $n \in \Z_{\ge 1}$ and 
use the notation 
$[i,j]=\{k \in \Z\mid i \le k \le j\}$, 
the characteristic function 
$\theta(\mathrm{true})=1, 
\theta(\mathrm{false}) =0$, the Kronecker delta
$\delta_{\alpha, \beta}= \delta^\alpha_\beta =
\delta^{\alpha_1,\ldots, \alpha_m}_{\beta_1,\ldots, \beta_m} = 
\prod_{j=1}^m\theta(\alpha_j=\beta_j)$,
$|\alpha | = \alpha_1+\cdots  + \alpha_m$
for arrays
$\alpha=(\alpha_1,\ldots, \alpha_m), \beta=(\beta_1,\ldots, \beta_m)$ 
of {\em any} length $m$,  
$[u]=\frac{q^u-q^{-u}}{q-q^{-1}}$, 
the $q$-Pochhammer symbol
$(z; q)_m = \prod_{j=1}^m(1-zq^{j-1})$, 
the $q$-factorial $(q)_m = (q;q)_m$ and 
the $q$-binomial 
$\binom{m}{k}_{\!q} = \theta(k \in [0,m])
\frac{(q)_m}{(q)_k(q)_{m-k}}$.

\section{Quantum $R$ matrix for symmetric tensor representations of 
$U_q(A^{(1)}_n)$}\label{sec:R}

\subsection{\mathversion{bold}Quantum $R$ matrix $R(z)$}\label{ss:R}
We assume that $q$ is generic.
The Drinfeld-Jimbo quantum affine algebra (without derivation) 
$U_q(A^{(1)}_n) = U_q(\widehat{sl}_{n+1})$ \cite{D,J} is generated by 
$e_i, f_i, k^{\pm 1}_i\, (i \in \Z/(n+1)\Z)$ satisfying the relations 
\begin{equation*}
k_i k^{-1}_i = k^{-1}_ik_i = 1,\; \; [k_i, k_j]=0,\;\;
k_i e_j = D_{i,j}e_j k_i,\;\; k_i f_j = D_{i,j}^{-1}f_j k_i,\; \;
[e_i,f_j] = \delta_{i,j} \frac{k_i-k^{-1}_i}{q-q^{-1}}
\end{equation*}
and the Serre relations.
Here $D_{i,j} = q^{2\delta_{i,j}-\delta_{i,j-1}-\delta_{i,j+1}}$ with 
$\delta_{i,j} = \theta(i-j \in (n+1)\Z)$.
It is a Hopf algebra 
with the coproduct $\Delta$  given by
\begin{align}
&\Delta k^{\pm 1}_i = k^{\pm 1}_i\otimes k^{\pm 1}_i,\quad
\Delta e_i = 1\otimes e_i + e_i \otimes k_i,\quad
\Delta f_i = f_i\otimes 1 + k^{-1}_i\otimes f_i. \label{Delta}
\end{align}
For $l \in \Z_{\ge 1}$, introduce the vector space $V_l$ whose basis is labeled with 
the set $B_l$ as
\begin{align}\label{BV}
B_l = \{\alpha=(\alpha_1,\ldots, \alpha_{n+1}) \in \Z_{\ge 0}^{n+1}
\mid |\alpha| = l\},\quad
V_l  = \bigoplus_{\alpha=(\alpha_1,\ldots, \alpha_{n+1})  \in B_l}\C
|\alpha_1,\ldots, \alpha_{n+1}\rangle.
\end{align}
We write $|\alpha_1,\ldots, \alpha_{n+1}\rangle$ simply as $|\alpha \rangle$.
The degree-$l$ symmetric tensor representation 
with {\em spectral parameter} $x$ 
$\pi_x^l: U_q(A^{(1)}_n) \rightarrow \mathrm{End}(V_l)$ 
is a finite dimensional irreducible representation given by
\begin{equation}\label{actsA}
\pi_x^l(e_i)|\alpha\rangle
= x^{\delta_{i,0}}[\alpha_i]|\alpha-\hat{i}\,\rangle,\quad
\pi_x^l(f_i)|\alpha\rangle
= x^{-\delta_{i,0}}[\alpha_{i+1}]|\alpha+\hat{i}\,\rangle,\quad
\pi_x^l(k_i)|\alpha\rangle
= q^{\alpha_{i+1}-\alpha_i}|\alpha\rangle,
\end{equation}
where $\hat{i}=(0,\ldots,0, 1,-1,0,\ldots,0)\in \Z^{n+1}$ 
contains $1,-1$ at the $i$-th and the $(i\!+\!1)$-th positions from the left 
and all the indices are to be understood mod $n+1$ as usual.
In (\ref{actsA}), vectors $|\alpha \pm \hat{i}\rangle$ 
such that $\alpha \pm \hat{i}\not\in B_l$
are to be understood as zero.

\begin{remark}\label{re:anti}
Let $U_q(A_n)$ be the subalgebra generated by $e_i, f_i, k^{\pm 1}_i$ with $i \neq 0$.
As a $U_q(A_n)$-module, the highest weight vector in $V_l$ is 
$|0,\ldots, 0, l\rangle$, which is also annihilated by all the 
$f_i$'s  except $f_n$. Thus $V_l$ is actually the 
$l$-fold symmetric tensor of the {\em anti-vector} representation
which corresponds to the $n \times l$ rectangular Young diagram.
\end{remark}

For generic $x$ and $y$, 
the tensor product representations 
$\pi_{x,y}^{l,m}:=(\pi_x^l\otimes \pi_y^m)\circ \Delta$ on $V_l\otimes V_m$  
is irreducible and isomorphic to $\pi_{y,x}^{m,l}$.
From this fact and (\ref{actsA}), it follows that there is a unique intertwiner
$\check{R}(z)=\check{R}^{l,m}(z): V_l \otimes V_m \rightarrow V_m \otimes V_l$ 
depending on $z=x/y$ satisfying 
\begin{align}\label{ir}
\check{R}(z) \pi_{x,y}^{l,m}(g) = 
\pi_{y,x}^{m,l}(g) \check{R}(z),\qquad \forall g \in U_q(A^{(1)}_n)
\end{align}
up to an overall normalization.
We fix it by
\begin{align}\label{nor}
\check{R}(z)(|0,\ldots, 0,l\rangle\otimes |0,\ldots, 0,m\rangle)
= |0,\ldots, 0,m\rangle\otimes |0,\ldots, 0,l\rangle.
\end{align}

Let us further introduce $R(z) = R^{l,m}(z)=P \check{R}^{l,m}(z) \in 
\mathrm{End}(V_l \otimes V_m)$, where 
$P(|\alpha\rangle \otimes |\beta\rangle) = 
|\beta\rangle \otimes |\alpha\rangle$ is the transposition.
The both $R(z)$ and $\check{R}(z)$ will be called the
{\em quantum $R$ matrix} or just $R$ matrix for short.
Its action is expressed as 
\begin{align}\label{RR}
R(z)(|\alpha\rangle \otimes | \beta\rangle ) = 
\sum_{\gamma,\delta}R(z)_{\alpha,\beta}^{\gamma,\delta}
|\gamma\rangle \otimes | \delta\rangle,\quad
\check{R}(z)(|\alpha\rangle \otimes | \beta\rangle ) = 
\sum_{\gamma,\delta}R(z)_{\alpha,\beta}^{\gamma,\delta}
|\delta\rangle \otimes | \gamma\rangle,
\end{align}
where $\alpha \in B_l, \beta\in B_m$ and the 
sums are taken over 
$\gamma \in B_l, \delta \in B_m$.
The matrix elements $R(z)_{\alpha,\beta}^{\gamma,\delta}$ 
are rational functions in $z$ and $q$.
In principle, they are computable either by the {\em fusion}  \cite{KRS} 
from the $(l,m)=(1,1)$ case (bottom-up)
or by taking the image of the {\em universal $R$} (top-down).
Practically an efficient alternative is to evaluate the 
trace of the product of the {\em three-dimensional $R$ operators} 
\cite{KV, BS, BMS, KO1} satisfying the {\em tetrahedron equation}.
This approach has been developed in \cite{BS, M1, M2, KS, KO2, KOS}
as an outgrowth of the pioneering works \cite{Zam80, BB, SMS}.
Examples in 
Appendix \ref{app:R} have been generated by this method by using 
\cite[eq.(2.24)]{KOS}$|_{\epsilon_1=\cdots = \epsilon_n = 0}$.
See also \cite{KMO2} for a recent application of 
the tetrahedron equation to a multispecies TAZRP.

We depict the matrix element of the $R$ matrix as
\begin{equation}\label{vertex}
\begin{picture}(200,45)(-90,-21)
\put(-60,-2){$R(z)_{\alpha, \beta}^{\gamma,\delta}=$}
\put(0,0){\vector(1,0){24}}
\put(12,-12){\vector(0,1){24}}
\put(-10,-2){$\alpha$}\put(27,-2){$\gamma$}
\put(9,-22){$\beta$}\put(9,16){$\delta$}
\end{picture}
\end{equation}
suppressing dependence on $n,z,q$, and also $l,m$ associated with 
the horizontal and vertical lines, respectively.
This picture matches the action of $\check{R}(z)$ in (\ref{RR})
viewed in the $\nearrow$ direction. 
The relation (\ref{ir}) with $g=k_i$ tells 
the {\em weight conservation} property that
$R(z)_{\alpha, \beta}^{\gamma,\delta}=0$ unless 
$\alpha+\beta=\gamma+\delta \in \Z_{\ge 0}^{n+1}$.

The most significant property of the $R$ matrix is the 
Yang-Baxter equation \cite{Bax} which is presented in two equivalent forms:
\begin{align}
&(\check{R}^{l,m}(x)\otimes 1)
(1 \otimes \check{R}^{k,m}(xy))(\check{R}^{k,l}(y)\otimes 1)
= (1 \otimes \check{R}^{k,l}(y))(\check{R}^{k,m}(xy)\otimes 1)
(1 \otimes \check{R}^{l,m}(x)),\label{yb1}\\
&R_{2,3}^{l,m}(y)R_{1,3}^{k,m}(xy)R_{1,2}^{k,l}(x)
= R_{1,2}^{k,l}(x) R_{1,3}^{k,m}(xy)R_{2,3}^{l,m}(y),\label{yb2}
\end{align}
where the lower 
indices in (\ref{yb2}) specify the components 
on which $R(z)$ acts nontrivially\footnote{
Although subtle, 
we distinguish the degrees $l,m$ of symmetric tensors, 
components $i,j$ in tensor products in $R^{l,m}_{i,j}(z)$ from the 
indices $\alpha, \beta, \gamma, \delta$ specifying the element
$R(z)^{\gamma, \delta}_{\alpha, \beta}$ by 
putting them on the opposite side of the spectral parameter $(z)$.
The similar convention will be used also for 
$S(z)$ and $\mathscr{S}(\lambda, \mu)$
introduced later.}.
The relations (\ref{yb1}) and (\ref{yb2})  hold as the operators 
$V_k \otimes V_l \otimes V_m \rightarrow V_m \otimes V_l \otimes V_k$ and 
$V_k \otimes V_l \otimes V_m \rightarrow V_k \otimes V_l \otimes V_m$, respectively.
The equality of the matrix element for 
$|\alpha\rangle \otimes |\beta\rangle \otimes |\gamma\rangle
\mapsto 
|\alpha''\rangle \otimes |\beta''\rangle \otimes |\gamma''\rangle$
in (\ref{yb2}) is depicted as
\begin{equation}\label{ybe}
\begin{picture}(150,105)(60,0)

\put(0,0){
\put(-50,50){${\displaystyle \sum_{\alpha',\beta',\gamma'}}$}
\put(0,58){\vector(3,-1){80}}
\put(-13,58){$\alpha$}\put(35,32){$\alpha'$}\put(85,27){$\alpha''$}
\put(0,42){\vector(3,1){80}}
\put(-13,37){$\beta$}\put(37,59){$\beta'$}\put(85,67){$\beta''$}
\put(60,15){\vector(0,1){70}}
\put(58,5){$\gamma$}\put(63,48){$\gamma'$}\put(58,91){$\gamma''$}
\put(120,50){$=$}}

\put(190,0){
\put(-45,50){${\displaystyle \sum_{\alpha',\beta',\gamma'}}$}
\put(0,70){\vector(3,-1){80}}
\put(0,15){\put(-10,55){$\alpha$}\put(35,47){$\alpha'$}\put(85,25){$\alpha''$}}
\put(0,30){\vector(3,1){80}}
\put(0,-12){\put(-11,37){$\beta$}\put(37,46){$\beta'$}\put(85,67){$\beta''$}}
\put(20,15){\vector(0,1){70}}
\put(-40,0){\put(58,5){$\gamma$}\put(49,48){$\gamma'$}\put(58,91){$\gamma''$}}}

\end{picture}
\end{equation}
The $R$ matrix also satisfies
\begin{align}
&\check{R}^{l,m}(z)\check{R}^{m,l}(z^{-1}) = \mathrm{id}_{V_m \otimes V_l},
\label{inv}\\
&R(z)^{\gamma, \delta}_{\alpha, \beta} = 
R(z)_{\gamma', \delta'}^{\alpha', \beta'}
\prod_{i=1}^{n+1}\frac{(q^2)_{\alpha_i}(q^2)_{\beta_i}}
{(q^2)_{\gamma_i}(q^2)_{\delta_i}}. 
\label{rev}
\end{align}
The former is called the inversion relation. 
In the latter $\alpha'=(\alpha_{n+1},\ldots, \alpha_1)$ denotes the 
reverse array of $\alpha=(\alpha_1,\ldots, \alpha_{n+1})$ and 
$\beta', \gamma', \delta'$ are similarly defined. 
It is a corollary of  \cite[eqs. (2.4), (2.24)]{KOS}.

\begin{theorem}\label{pr:Rs}
For $l\le m$, 
elements of the $R$ matrix $R(z)=R^{l,m}(z)$ 
admit the explicit formula at $z=q^{l-m}$:
\begin{align}
R(q^{l-m})_{\alpha, \beta}^{\gamma,\delta}
&= \delta_{\alpha+\beta}^{\gamma+\delta}\,
q^\psi
\binom{m}{l}_{\! q^2}^{\!-1}\,
\prod_{i=1}^{n+1}\binom{\beta_i}{\gamma_i}_{\! q^2},
\label{rs}\\
\psi&= \psi_{\alpha, \beta}^{\gamma,\delta}
= \sum_{1 \le i<j \le n+1}\alpha_i(\beta_j-\gamma_j)+\sum_{1 \le i<j \le n+1}(\beta_i-\gamma_i)\gamma_j.\label{psi}
\end{align}
\end{theorem}
Note that the $q$-binomial factors in (\ref{rs}) tell that 
$R(q^{l-m})_{\alpha, \beta}^{\gamma,\delta}=0$ unless 
$\beta\ge \gamma$ or equivalently 
$\alpha \le \delta$ under the condition 
$\alpha+\beta= \gamma+\delta$.
Here and in what follows, 
$u\ge v$ for $u,v \in \Z^k$ for any $k$ is defined by 
$u-v \in \Z_{\ge 0}^k$ and $\le$ is defined similarly.
The condition $l\le m$ in the claim matches this property.
It is interesting that 
the  ``inter-color coupling" enters only via $\psi$ apparently.
See the end of Appendix \ref{app:R}  for an example.
For the proof we prepare 
\begin{lemma}\label{le:psi}
For any $i \in \Z/(n+1)\Z$, the following equalities are valid:
\begin{align*}
\psi_{\alpha, \beta}^{\gamma-\hat{i},\delta} - 
\psi_{\alpha, \beta}^{\gamma,\delta-\hat{i}}
&= \gamma_{i+1}-\alpha_i+\beta_i-\gamma_i+1+(l-m)\delta_{i,0},\\
\psi_{\alpha+\hat{i}, \beta}^{\gamma,\delta}-
\psi_{\alpha, \beta}^{\gamma,\delta-\hat{i}}
&= \beta_{i+1}-\gamma_{i+1}+(l-m)\delta_{i,0},\\
\psi_{\alpha, \beta+\hat{i}}^{\gamma,\delta}-
\psi_{\alpha, \beta}^{\gamma,\delta-\hat{i}}
&= \gamma_{i+1}-\alpha_i.
\end{align*}
\end{lemma}
\begin{proof}
A direct calculation.
\end{proof}

{\em Proof of Theorem \ref{pr:Rs}}.
$R(z)$ is not singular at $z=q^{l-m}$.
See for example \cite[eq.(6.16)]{KOS}.
Thus it suffices to check that the RHS of (\ref{rs}) satisfies (\ref{ir}) and (\ref{nor}).
The latter is obvious.
The relation (\ref{ir}) with $g=k_i$ 
means the weight conservation and it holds due to the factor 
$\delta_{\alpha+\beta}^{\gamma+\delta}$.
In the sequel we show (\ref{ir}) for $g=f_i$. 
The case $g=e_i$ can be verified similarly.
Let the both sides of (\ref{ir}) 
act on $|\alpha\rangle \otimes | \beta \rangle \in V_l \otimes V_m$
and compare the coefficients of $|\delta\rangle \otimes |\gamma\rangle$
in the output vector.
Using (\ref{Delta}), (\ref{actsA}) and (\ref{RR}) 
we find that the relation to be proved is 
\begin{equation*}
\begin{split}
&R(z)_{\alpha,\beta}^{\gamma,\delta-\hat{i}}[\delta_{i+1}+1]\theta(\delta_i \ge 1)
+R(z)_{\alpha,\beta}^{\gamma-\hat{i},\delta}
q^{\delta_i-\delta_{i+1}}z^{-\delta_{i,0}}
[\gamma_{i+1}+1]\theta(\gamma_i \ge 1)\\
&= R(z)_{\alpha+\hat{i},\beta}^{\gamma,\delta}[\alpha_{i+1}]z^{-\delta_{i,0}}
+ R(z)_{\alpha,\beta+\hat{i}}^{\gamma,\delta}[\beta_{i+1}]q^{\alpha_i-\alpha_{i+1}}
\end{split}
\end{equation*}
at $z=q^{l-m}$ 
under the weight conservation condition (i)
$\alpha_i+\beta_i=\gamma_i+\delta_i-1$ and (ii)
$\alpha_{i+1}+\beta_{i+1}= \gamma_{i+1}+\delta_{i+1}+1$.
By substituting (\ref{rs}) and applying Lemma \ref{le:psi}, this is simplified to
\begin{equation*}
\begin{split}
&[\delta_{i+1}+1](1-q^{2\beta_{i+1}})
(1-q^{2(\beta_i-\gamma_i+1)})(1-q^{2(\gamma_{i+1}+1)})
\theta(\delta_i \ge 1)\\
&+q^{\gamma_{i+1}-\delta_{i+1}+2\beta_i-2\gamma_i+2}
[\gamma_{i+1}+1](1-q^{2\beta_{i+1}})
(1-q^{2(\beta_{i+1}-\gamma_{i+1})})(1-q^{2\gamma_i})
\theta(\gamma_{i}\ge 1)\\
&= q^{\beta_{i+1}-\gamma_{i+1}}[\alpha_{i+1}]
(1-q^{2\beta_{i+1}})(1-q^{2(\beta_i-\gamma_i+1)})(1-q^{2(\gamma_{i+1}+1)})\\
&+ q^{\gamma_{i+1}-\alpha_{i+1}}[\beta_{i+1}]
(1-q^{2(\beta_i+1)})(1-q^{2(\beta_{i+1}-\gamma_{i+1})})
(1-q^{2(\gamma_{i+1}+1)}).
\end{split}
\end{equation*} 
We may drop $\theta(\delta_i\ge 1)$ because if $\delta_i=0$, 
the weight condition (i) $\alpha_i+\beta_i-\gamma_i+1=0$ enforces 
$1-q^{\beta_i-\gamma_i+1}=0$. 
Similarly $\theta(\gamma_i\ge 1)$ can also be discarded.
Then we are left to show
\begin{align*}
&(1-q^{2(\delta_{i+1}+1)})(1-q^{2(\beta_i-\gamma_i+1)})
+q^{2+2\beta_i-2\gamma_i}(1-q^{2\gamma_i})(1-q^{2(\beta_{i+1}-\gamma_{i+1})})\\
&= q^{2(\beta_{i+1}-\gamma_{i+1})}(1-q^{2\alpha_{i+1}})
(1-q^{2(\beta_{i}-\gamma_{i}+1)})
+(1-q^{2(\beta_i+1)})(1-q^{2(\beta_{i+1}-\gamma_{i+1})}).
\end{align*}
This is easily checked by using the weight condition (ii).
\qed

\subsection{\mathversion{bold}Stochastic $R$ matrix $S(z)$} 
We introduce a slight but essential modification 
$S(z) = S^{l,m}(z) \in \mathrm{End}(V_l\otimes V_m)$ of the $R$ matrix by
\begin{align}
 &S(z) (|\alpha\rangle \otimes | \beta\rangle ) = 
\sum_{\gamma,\delta}S(z)_{\alpha,\beta}^{\gamma,\delta}
|\gamma\rangle \otimes | \delta\rangle,\quad
S(z)^{\gamma,\delta}_{\alpha, \beta} 
= q^\eta R(z)^{\gamma,\delta}_{\alpha, \beta},
\label{SR}\\
&\eta=\eta_{\alpha,\beta}^{\gamma,\delta} 
= \sum_{1 \le i<j \le n+1}(\beta_i - \gamma_i)\gamma_j - 
\sum_{1 \le i<j \le n+1}
\alpha_i(\beta_j-\gamma_j)
=
\sum_{1 \le i<j \le n+1}
(\delta_i\gamma_j - \alpha_i\beta_j),\label{eta}
\end{align}
where the sum $\sum_{\gamma,\delta}$ is taken over 
$\gamma \in B_l, \delta \in B_m$ as in (\ref{RR}).
The last equality in (\ref{eta}) is derived by using 
$\alpha_i+\beta_i = \gamma_i + \delta_i$. 
We also introduce $\check{S}(z) = P S(z)$.
The both $S(z)$ and $\check{S}(z)$ 
will be called the {\em stochastic $R$ matrix} or 
just $S$ matrix for short.

\begin{proposition}\label{pr:S}
The $S$ matrix satisfies the inversion relation
$\check{S}^{l,m}(z)\check{S}^{m,l}(z^{-1}) = \mathrm{id}_{V_m \otimes V_l}$
and the Yang-Baxter equation 
$S_{2,3}^{l,m}(y)S_{1,3}^{k,m}(xy)S_{1,2}^{k,l}(x)
= S_{1,2}^{k,l}(x) S_{1,3}^{k,m}(xy)S_{2,3}^{l,m}(y)$.
\end{proposition}
\begin{proof}
The inversion relation is obvious.
Consider the Yang-Baxter equation 
depicted in (\ref{ybe}).
In view of the last expression in (\ref{eta}) we concern the sum of the three 
$\eta$'s on each side:
\begin{align*}
X&= \beta_i'\alpha'_j - \alpha_i\beta_j +
\gamma''_i\beta''_j - \beta'_i\gamma'_j +
\gamma'_i \alpha''_j - \alpha'_i \gamma_j,\\
Y&= \gamma'_i \beta'_j - \beta_i\gamma_j +
\gamma''_i\alpha'_j - \alpha_i \gamma'_j +
\beta''_i\alpha''_j - \alpha'_i\beta'_j.
\end{align*}
It suffices to check 
(i) $X$  and $Y$ are independent of 
$\alpha', \beta', \gamma'$,
(ii)  $X=Y$.
The both are easy to verify by using the weight conservation condition.
\end{proof}

\begin{lemma}\label{chu}
For $U_q(A^{(1)}_1)$, the following relation is valid:
\begin{align*}
(\Delta f_1)^{s} (|0,A\rangle \otimes |0,B\rangle)
&= F(s, A+B)\sum_{a_1+b_1=s}
q^{a_1b_2}\binom{A}{a_1}_{\!q^2}\binom{B}{b_1}_{\!q^2}
|a_1,a_2\rangle \otimes |b_1,b_2\rangle,
\end{align*}
where $F$ is a known function and  
$a_2, b_2$ are determined from $a_1, b_1$ by 
$a_1+a_2=A, b_1+b_2=B$. 
\end{lemma}
\begin{proof}
From $k_1f_1 = q^{-2}f_1k_1$, we get
\begin{align*}
(\Delta f_1)^{s} (|0,A\rangle \otimes |0,B\rangle)
&= \sum_{a_1+b_1=s}
\binom{s}{a_1}_{\!q^2}f^{a_1}_1k_1^{-b_1}|0,A\rangle
\otimes f_1^{b_1} |0,B\rangle\\
&=\sum_{a_1+b_1=s}\binom{s}{a_1}_{\!q^2}
\frac{[A]![B]!}{[a_2]![b_2]!}
q^{-(a_1+a_2)b_1}|a_1,a_2\rangle \otimes |b_1,b_2\rangle,
\end{align*}
where $[m]!= [m][m-1]\cdots [1] 
= \frac{q^{-m(m-1)/2}(q^2)_m}{(1-q^2)^{m}}$.
The last coefficient equals 
$q^\omega\binom{s}{a_1}_{\!q^2}
\frac{(q^2)_{A}(q^2)_{B}}
{(q^2)_{a_2}(q^2)_{b_2}}\frac{1}{(1-q^2)^{s}}$ with the power
$\omega$ given by
\begin{align*}
\omega &= -\frac{(a_1+a_2)(a_1+a_2-1)}{2}+\frac{a_2(a_2-1)}{2}
-\frac{(b_1+b_2)(b_1+b_2-1)}{2}+\frac{b_2(b_2-1)}{2}
-(a_1+a_2)b_1\\
&= -\frac{(a_1+b_1)(a_1+b_1-1)}{2}
-(a_1+b_1)(a_2+b_2)+a_1b_2.
\end{align*}
Since $a_2+b_2=A+B-s$, $\omega$ is a function of 
$s$ and $A+B$ except the last term
$a_1b_2$.
\end{proof}
The most notable feature of the $S$ matrix is the following.
\begin{theorem}\label{pr:s=1}
For any $l,m \in \Z_{\ge 1}$,
the $S$ matrix 
$S(z) = S^{l,m}(z)$ enjoys the sum-to-unity property:
\begin{align}\label{msk}
\sum_{\gamma \in B_l, \delta \in B_m}
S(z)^{\gamma,\delta}_{\alpha, \beta} =1,\qquad
\forall (\alpha, \beta) \in B_l \times B_m.
\end{align}
\end{theorem}
Note that there is no constraint $l\le m$ for this assertion.
\begin{proof}
We are to show
$\sum_{\gamma,\delta}q^{\sum_{i<j}(\delta_i\gamma_j-\alpha_i\beta_j)}
R(z)^{\gamma,\delta}_{\alpha, \beta} =1$.
By means of (\ref{rev}),  the relation (\ref{msk}) is rewritten as
\begin{align}\label{mp}
\sum_{\gamma,\delta}
\frac{q^{\sum_{i<j}\gamma_i\delta_j}}{\prod_i(q^2)_{\gamma_i}(q^2)_{\delta_i}}
R(z)_{\gamma,\delta}^{\alpha, \beta} = 
\frac{q^{\sum_{i>j}\alpha_i\beta_j}}{\prod_i(q^2)_{\alpha_i}(q^2)_{\beta_i}},
\end{align}
where $\sum_{i<j}=\sum_{1 \le i < j \le n+1}$,
$\sum_{i>j}=\sum_{1 \le j < i \le n+1}$, 
$\prod_i = \prod_{1\le i \le n+1}$ and 
$\sum_{\gamma,\delta}$ is taken over 
$(\gamma,\delta) \in B_l \times B_m$.
Summing 
$(\ref{mp})\times (|\beta\rangle \otimes |\alpha\rangle)$ 
over $\alpha \in B_l, \beta \in B_m$
satisfying $\alpha+\beta=r$
for a fixed $r=(r_1,\ldots, r_{n+1}) \in \Z_{\ge 0}^{n+1}$, we get 
\begin{align*}
\sum_{\alpha+\beta=r}
\sum_{\gamma,\delta}
\frac{q^{\sum_{i<j}\gamma_i\delta_j}}{\prod_i(q^2)_{\gamma_i}(q^2)_{\delta_i}}
R(z)_{\gamma,\delta}^{\alpha, \beta}
|\beta\rangle \otimes |\alpha\rangle
=\sum_{\alpha+ \beta=r}
\frac{q^{\sum_{i>j}\alpha_i\beta_j}}{\prod_i(q^2)_{\alpha_i}(q^2)_{\beta_i}}
|\beta\rangle \otimes |\alpha\rangle.
\end{align*}
This is neatly expressed as
\begin{align*}
\check{R}(z) w^{(r)}_{l,m} = w^{(r)}_{m,l},\quad
\text{where}\quad
w^{(r)}_{l,m}= \sum_{\lambda \in B_l, \kappa\in B_m, \lambda+\kappa=r}
\frac{q^{\sum_{i<j}\lambda_i\kappa_j}}{\prod_i(q^2)_{\lambda_i}(q^2)_{\kappa_i}}
|\lambda\rangle \otimes | \kappa\rangle 
\in V_l \otimes V_m.
\end{align*}
It follows from (\ref{nor}) 
by applying $(\Delta f_1)^{r_1}
(\Delta f_2)^{r_1+r_2}\cdots (\Delta f_n)^{r_1+\cdots +r_n}$ successively
using the commutativity 
$\pi_{y,x}^{m,l}(f_i) \check{R}(z) = \check{R}(z) 
\pi_{x,y}^{l,m}(f_i)$ and 
Lemma \ref{chu}.
\end{proof}
The above proof elucidates that the sum-to-unity relations are
nothing but the $U_q(A_n)$-orbit of the unit normalization condition (\ref{nor}). 

For $\beta=(\beta_1, \ldots, \beta_n), 
\gamma = (\gamma_1,\ldots, \gamma_n) \in \Z_{\ge 0}^n$,
we define\footnote{
We will adequately mention $\Z^n$ or $\Z^{n+1}$ to avoid confusion 
and prefer to use the simpler notation $\beta$ etc. 
than bothering by writing $\bar{\beta}$ etc. except 
the inevitable coexistence within a formula like (\ref{ycn}).}
\begin{align}
\Phi_q(\gamma|\beta; \lambda,\mu) &= 
q^\xi \left(\frac{\mu}{\lambda}\right)^{|\gamma|}
\frac{(\lambda;q)_{|\gamma|}(\frac{\mu}{\lambda};q)_{|\beta|-|\gamma|}}
{(\mu;q)_{|\beta|}}
\prod_{i=1}^{n}\binom{\beta_i}{\gamma_i}_{\!q},\label{Pdef}\\
\xi&= \xi_{\beta,\gamma} = \sum_{1 \le i<j \le n} (\beta_i-\gamma_i)\gamma_j,
\label{xi}
\end{align}
where $\lambda, \mu$ are generic parameters.
By the definition $\Phi_q(\gamma|\beta;\lambda, \mu) =0$ 
unless $\gamma \le \beta$.
Note that $\beta$ and $\gamma$ here are 
$n$-component arrays rather than $n+1$ as opposed to the indices in 
$S(z)_{\gamma,\delta}^{\alpha, \beta}$.
In the case $n=1$, the power $\xi$ vanishes and 
the function (\ref{Pdef}) reproduces \cite[eq.(8)]{P}  as
\begin{align}
\varphi(m|m') = \Phi_q(m|m';{\textstyle \frac{\nu}{\mu}}, \nu)|_{n=1},
\end{align}
which is known as 
the weight function associated with $q$-Hahn polynomials.
As it turns out, 
our $U_q(A^{(1)}_n)$ generalization (\ref{Pdef}) arises as the special value of the
$S$ matrix.

\begin{proposition} \label{pr:sp}
Suppose $l \le m$.  
Given $\beta=(\beta_1,\ldots, \beta_{n+1}) \in B_m$ and 
$\gamma = (\gamma_1,\ldots, \gamma_{n+1}) \in B_l$,
set $\bar{\beta}=(\beta_1,\ldots, \beta_n)$ and 
$\bar{\gamma} = (\gamma_1,\ldots, \gamma_n)$.
Then elements of the $S$ matrix $S(z) = S^{l,m}(z)$ at $z=q^{l-m}$ 
are given by
\begin{align}\label{ycn}
S(z=q^{l-m})^{\gamma,\delta}_{\alpha, \beta}
= \delta_{\alpha+\beta}^{\gamma+\delta}\,
\Phi_{q^2}(\bar{\gamma}| \bar{\beta}; q^{-2l},q^{-2m}).
\end{align}
\end{proposition}
\begin{proof}
Theorem \ref{pr:Rs} and (\ref{SR}) lead to
\begin{align}
S(z=q^{l-m})^{\gamma,\delta}_{\alpha, \beta}
&= \delta_{\alpha+\beta}^{\gamma+\delta}\,
q^{\eta+\psi}
\binom{m}{l}_{\!q^2}^{-1}
\prod_{i=1}^{n+1}\binom{\beta_i}{\gamma_i}_{\!q^2}.\label{ikt}
\end{align}
Using (\ref{psi}), (\ref{eta}), (\ref{xi}), 
$l= |\alpha | = |\bar{\alpha}|+\alpha_{n+1} = |\bar{\gamma}|+\gamma_{n+1}$ and 
$m= |\beta | =  |\bar{\beta}|+\beta_{n+1}$
we find 
\begin{align*}
\eta+\psi &=2 \sum_{1 \le i<j \le n+1}(\beta_i-\gamma_i)\gamma_j
= 2(|\bar{\beta}|-|\bar{\gamma}|)(l-|\bar{\gamma}|)+ 2\xi.
\end{align*}
On the other hand the two of the
$q$-binomial factors in (\ref{ikt}) are combined as
\begin{align*}
&\binom{m}{l}_{\!q^2}^{-1}
\binom{\beta_{n+1}}{\gamma_{n+1}}_{\!q^2}
=
\binom{m}{l}_{\!q^2}^{-1}
\binom{m- |\bar{\beta}|}{l- |\bar{\gamma}|}_{\!q^2}
=q^\phi\frac{(q^{2l-2m};q^2)_{|\bar{\beta}|-|\bar{\gamma}|}
(q^{-2l}; q^2)_{|\bar{\gamma}|}}
{(q^{-2m}; q^2)_{|\bar{\beta}|}},\\
&\phi = |\bar{\gamma}|(2l-|\bar{\gamma}|+1)
+(|\bar{\beta}|-|\bar{\gamma}|)(2m-2l-|\bar{\beta}|+|\bar{\gamma}|+1)
-|\bar{\beta}|(2m-|\bar{\beta}|+1).
\end{align*}
Thus the proof is finished by checking
$\eta+\psi+\phi = 2\xi+2(l-m)|\bar{\gamma}|$,
which is straightforward.
\end{proof}

In view of Proposition \ref{pr:sp}, 
Theorem \ref{pr:s=1} is rephrased in terms of an
$n$-component array $\beta$ as the identity
\begin{align}\label{pwa}
\sum_{\gamma \in \Z_{\ge 0}^n, |\gamma|\le l}
\Phi_q(\gamma |\beta; q^{-l},q^{-m}) = 1
\quad \text{for any}\;
\beta \in \Z_{\ge 0}^n\;\text{satisfying}\; |\beta| \le m
\end{align}
for any positive integers $l,m$ such that $l\le m$. 
One may remove the constraint $|\gamma|\le l$ in the sum since 
the summand vanishes otherwise.

\subsection{\mathversion{bold}Regarding $\lambda=q^{-l},\mu=q^{-m}$ 
as parameters}\label{ss:ok}

Proposition \ref{pr:S}, Theorem \ref{pr:s=1} and Proposition \ref{pr:sp} 
remain valid even when we
replace $q^{-l}$ and $q^{-m}$ with parameters $\lambda$ and $\mu$ as we shall 
explain below. 
In this subsection, we fix $q,z$, set $\lambda=q^{-l}, \mu=q^{-m}$ and
regard $\lambda, \mu$ as variables. Note that the action of 
$e_i,f_i,k_i^{\pm1}\in U_q(A^{(1)}_n)$ 
on $V_l\ot V_m$ gives rise to Laurent polynomials in $\lambda, \mu$.
We wish to show that the matrix elements $R(z)_{\alpha,\beta}^{\gamma,\delta}$
are rational functions in $\lambda,\mu$. 
Since $l$ varies, we utilize $\alpha=(\alpha_1,\ldots,\alpha_n)\in\Z_{\ge0}^n$
as a labeling of basis vectors $|\alpha_1,\ldots,\alpha_n\rangle$ of $V_l$. 
So is $\beta$ for $V_m$.
Thus the symbol $|0\rangle$ which is the abbreviation of 
$|0,\ldots,0\rangle$ is to be understood as an appropriate highest weight vector 
appearing in (\ref{nor}).
Due to the weight conservation property 
$R(z)_{\alpha,\beta}^{\gamma,\delta}=0$ unless $\alpha+\beta=\gamma+\delta$,
we concentrate on the case when $\alpha+\beta=\gamma+\delta=\varpi$ 
for some fixed weight $\varpi\in\Z_{\ge 0}^n$.  Take $N$ such that $|\varpi|<N$ and
then take $l,m$ such that $N<l,m$. 
Since $V_l\ot V_m$ is known to be irreducible over $U_q(A^{(1)}_n)$, there exist
elements $g_j\in U_q(A^{(1)}_n)$ ($j=1,\ldots,t;t=\prod_{i=1}^n(\varpi_i+1)$) such that
$\{\pi_{x,y}^{l,m}(g_j)(|0\rangle\ot|0\rangle)\mid j=1,\ldots,t\}$ spans the vector subspace
$\C\langle |\alpha\rangle\ot|\beta\rangle\mid \alpha+\beta=\varpi\rangle$ of
$V_l\ot V_m$ of weight $\varpi$. From the intertwining property \eqref{ir}, we have
\[
\check{R}(z)\pi^{l,m}_{x,y}(g_j)(|0\rangle\ot|0\rangle)
=\pi_{y,x}^{m,l}(g_j)(|0\rangle\ot|0\rangle)
\quad\text{for }j=1,\ldots,t.
\]
Here we have used the normalization \eqref{nor}. Solving the above linear equation
for $\{\check{R}(z)(|\alpha\rangle\ot|\beta\rangle)\mid \alpha+\beta=\varpi\}$,
one finds that the matrix coefficients 
$R(z)_{\alpha,\beta}^{\gamma,\delta}$ with the standard bases 
$\{|\alpha\rangle\ot|\beta\rangle\mid \alpha+\beta=\varpi\}$
are expressed by rational functions in $\lambda, \mu$.

Once we understand that $R(z)_{\alpha,\beta}^{\gamma,\delta}$ is a rational
function in $\lambda=q^{-l}, \mu=q^{-m}$, we can show that the Yang-Baxter equation 
\eqref{yb1} or \eqref{yb2} is satisfied as an identity of matrix-valued rational functions
in $\kappa=q^{-k},\lambda=q^{-l},\mu=q^{-m}$. 
To see this, fix a weight $\varpi=\alpha+\beta+\gamma$ and take an integer $N$
such that $|\varpi|<N$. Consider a particular coefficient of both sides 
of \eqref{yb2} applied to a vector 
$|\alpha\rangle\ot|\beta\rangle\ot|\gamma\rangle$ such that 
$\alpha+\beta+\gamma=\varpi$. Eliminating the denominators, both sides are
polynomials in $\kappa,\lambda, \mu$. We know that substituting 
$\kappa=q^{-k},\lambda=q^{-l},\mu=q^{-m}$ where $k,l,m$ are integers such that $N<k,l,m$, both sides
are equal to each other. Since we can choose infinitely many independent integers
for $k,l,m$, this identity must be the one as polynomials in $\kappa,\lambda, \mu$.

\subsection{\mathversion{bold}Specialized $S$ matrix
$\mathscr{S}(\lambda,\mu)$}

Based on the argument in Section \ref{ss:ok},
we move onto the situation where 
the positive integers $l, m$ are effectively replaced by 
continuous parameters $\lambda, \mu$.
We will work with the $n$-component arrays 
$\alpha=(\alpha_1,\ldots, \alpha_n) $ rather than the $(n+1)$-component ones in 
(\ref{BV}).
Set 
\begin{align*}
W = \bigoplus_{(\alpha_1,\ldots, \alpha_n) \in \Z_{\ge 0}^n}
\C|\alpha_1, \ldots, \alpha_n \rangle.
\end{align*}
The vector $|\alpha_1, \ldots, \alpha_n \rangle$ will simply be denoted by 
$|\alpha\rangle$\footnote{Note a slight notational change from 
Section \ref{ss:R} where $(n+1)$-component arrays are used as in (\ref{BV}).}.
Define the operator $\mathscr{S}(\lambda,\mu) \in \mathrm{End}(W \otimes W)$ by
\begin{align}
&\mathscr{S}(\lambda,\mu)(|\alpha\rangle \otimes | \beta\rangle ) = 
\sum_{\gamma,\delta \in 
\Z_{\ge 0}^n}\mathscr{S}(\lambda,\mu)_{\alpha,\beta}^{\gamma,\delta}
|\gamma\rangle \otimes | \delta\rangle,
\label{smdef}\\
&\mathscr{S}(\lambda,\mu)^{\gamma,\delta}_{\alpha, \beta} 
= \delta^{\gamma+\delta}_{\alpha+\beta}
\Phi_q(\gamma | \beta; \lambda,\mu), \label{lin}
\end{align}
where $\Phi_q(\gamma | \beta; \lambda,\mu)$ is specified by (\ref{Pdef}) and 
(\ref{xi}).
The sum (\ref{smdef}) is finite by the weight conservation.
In fact, the direct sum decomposition
$W \otimes W = \bigoplus_{\kappa \in \Z_{\ge 0}^n}
\left(\bigoplus_{\alpha+\beta=\kappa}\C
|\alpha\rangle \otimes | \beta\rangle\right)$ holds
and $\mathscr{S}(\lambda,\mu)$ splits into the corresponding submatrices. 
We set $\check{\mathscr{S}}(\lambda, \mu) = P \mathscr{S}(\lambda,\mu) 
\in \mathrm{End}(W \otimes W)$ and call 
$\mathscr{S}(\lambda, \mu)$ and $\check{\mathscr{S}}(\lambda, \mu)$
the specialized $S$ matrix.
From (\ref{SR}) and (\ref{ycn}), the relation
\begin{align}\label{taio}
\mathscr{S}(\lambda = q^{-l}, \mu = q^{-m})
=S^{l,m}(z=q^{l-m})|_{q\rightarrow q^{1/2}}
\end{align}
holds for $l, m \in \Z_{\ge 1}$ such that $l \le m$.
The specialized $S$ matrix $\mathscr{S}(\lambda, \mu)$ is 
an extrapolation of it into generic $l,m$.

It satisfies the Yang-Baxter equation, 
the inversion relation and the sum-to-unity condition:
\begin{align}
&\mathscr{S}_{1,2}(\nu_1,\nu_2)
\mathscr{S}_{1,3}(\nu_1, \nu_3)
\mathscr{S}_{2,3}(\nu_2, \nu_3)
=
\mathscr{S}_{2,3}(\nu_2, \nu_3)
\mathscr{S}_{1,3}(\nu_1, \nu_3)
\mathscr{S}_{1,2}(\nu_1,\nu_2),
\label{sybe}\\
&\check{\mathscr{S}}(\lambda, \mu)
\check{\mathscr{S}}(\mu,\lambda)
= \mathrm{id}_{W^{\otimes 2}},
\label{sinv}\\
&\sum_{\gamma, \delta \in \Z_{\ge 0}^n}
\mathscr{S}(\lambda,\mu)^{\gamma,\delta}_{\alpha, \beta} =1\qquad
(\forall \alpha, \beta \in \Z_{\ge 0}^n).
\label{psum}
\end{align}
They are consequences of 
Proposition \ref{pr:S}, Theorem \ref{pr:s=1} and the argument in Section \ref{ss:ok}.

\begin{remark}\label{re:dp}
As seen from (\ref{Pdef}) and (\ref{lin}), the specialized $S$ matrix 
$\mathscr{S}(\lambda, \mu)$ is a solution of the Yang-Baxter equation 
{\em without} ``difference property", meaning that 
its dependence on $\lambda$ and $\mu$ is   
not only through the combination $\lambda/\mu$.
\end{remark}

As a supplement we include a direct proof of 
(\ref{psum}), namely the identity
\begin{align}\label{syk}
\sum_{\gamma \in \Z_{\ge 0}^n, \gamma \le \beta}
\Phi_q(\gamma | \beta; \lambda,\mu)  = 1 \quad(\forall \beta \in \Z_{\ge 0}^n),
\end{align}
where the condition $\gamma \le \beta$ may be dropped but is 
exhibited for clarity in the argument below.
In terms of 
$\tilde{\Phi}^{(n)}_q(\gamma|\beta;\lambda,\mu):= q^\xi (\mu/\lambda)^{|\gamma|}
(\lambda; q)_{|\gamma|}(\mu/\lambda;q)_{|\beta|-|\gamma|}
\prod_{i=1}^{n}\binom{\beta_i}{\gamma_i}_{\!q}$,
the relation (\ref{syk}) reads 
\[
\sum_{\gamma \in \Z_{\ge 0}^n, \gamma \le \beta} 
\tilde{\Phi}^{(n)}_q(\gamma|\beta;\lambda,\mu) = (\mu;q)_{|\beta|}.
\]
We set $\nu= \mu/\lambda$.
The case $n=1$ is equivalent to 
$\sum_{j=0}^k \nu^{k-j}(\nu;q)_j 
\binom{k}{j}_{\!q}=1$ for $\forall k \in \Z_{\ge 0}$, which is easily verified.
We invoke the induction on $n$.
Define $\hat{\beta} = (\beta_2,\ldots, \beta_n)$
and similarly $\hat{\gamma}$. 
From (\ref{xi}) one has $\xi=(\beta_1-\gamma_1)|\hat{\gamma}|
+ \sum_{2 \le i<j \le n} (\beta_i-\gamma_i)\gamma_i$,
therefore the LHS is expressed as
\begin{align*}
&\sum_{\gamma_1 \le \beta_1}
\nu^{\gamma_1}(\lambda;q)_{\gamma_1}
(\nu;q)_{\beta_1-\gamma_1}
\binom{\beta_1}{\gamma_1}_{\!q}
\times 
\sum_{\hat{\gamma} \le \hat{\beta}}
\tilde{\Phi}^{(n-1)}_q(\hat{\gamma}|\hat{\beta}; \lambda q^{\gamma_1},
\mu q^{\beta_1})\\
&=\sum_{\gamma_1 \le \beta_1}
\nu^{\gamma_1}(\lambda;q)_{\gamma_1}
(\nu;q)_{\beta_1-\gamma_1}
\binom{\beta_1}{\gamma_1}_{\!q}
(\mu q^{\beta_1};q)_{|\hat{\beta}|} \\
&= (\mu;q)_{\beta_1} (\mu q^{\beta_1};q)_{|\hat{\beta}|} 
= (\mu;q)_{|\beta|},
\end{align*}
where the first and the second equalities are due to the induction assumption 
at $n=n-1$ and $n=1$, respectively.

\section{Stochastic models}\label{sec:sm}
In this and the next section, we will be exclusively concerned with systems with 
the periodic boundary condition.

\subsection{Commuting transfer matrices}
We construct two types of commuting transfer matrices 
based on the stochastic $R$ matrices $S(z)$ and $\mathscr{S}(\lambda,\mu)$.
To extract Markov processes from them
one has to find an appropriate 
specialization that fulfills the basic axioms of the Markov matrix.
This issue will be argued in Section \ref{ss:dmc1}, \ref{ss:dmc2} and \ref{ss:ctm}.

First consider the $S$ matrix $S^{l,m}(z)$ with 
positive integers $l$ and $m$.  
For $l,m_1,\ldots, m_L \in \Z_{\ge 1}$ and parameters $z, w_1, \ldots, w_L$,
set
\begin{align}\label{tdef}
T(l,z|{\textstyle {m_1,\ldots, m_L \atop w_1,\ldots, w_L})}= \mathrm{Tr}_{V_l}\left(
S^{l,m_L}_{0,L}(z/w_L)\cdots S^{l,m_1}_{0,1}(z/w_1)\right)
\in \mathrm{End}\left(V_{m_1} \otimes \cdots \otimes V_{m_L}\right).
\end{align}
In the terminology of the quantum inverse scattering method,
it is the row transfer matrix of the $U_q(A^{(1)}_n)$ vertex model 
of length $L$ with periodic boundary condition 
whose quantum space is  
$V_{m_1} \otimes \cdots \otimes V_{m_L}$ with inhomogeneity parameters 
$w_1, \ldots, w_L$ and the auxiliary space $V_l$ signified by 0 
with spectral parameter $z$.
The $S^{l,m_i}_{0,i}(z/w_i)$ is the $S$ matrix (\ref{SR}) acting as 
$S^{l,m_i}(z/w_i)$ on $V_l \otimes V_{m_i}$ and as identity elsewhere.
The dependence on $q$ has been suppressed in the notation.
Note the obvious property
$T(l,z|{\textstyle {m_1,\ldots, m_L \atop w_1,\ldots, w_L})}=
T(l, az|{\textstyle {m_1,\ldots, m_L \atop aw_1,\ldots, aw_L})}$
for any $a$.

Thanks to Proposition \ref{pr:S} and the general principle \cite{Bax}, 
it forms a commuting family:
\begin{align}\label{ask}
[T(l,z|{\textstyle {m_1,\ldots, m_L \atop w_1,\ldots, w_L})} , 
T(l',z'|{\textstyle {m_1,\ldots, m_L \atop w_1,\ldots, w_L})}]=0.
\end{align}
We write the action of 
$T=T(l,z|{\textstyle {m_1,\ldots, m_L \atop w_1,\ldots, w_L})}$ 
on the vector representing a row configuration
as\footnote{
We warn that 
$|\alpha_1,\ldots, \alpha_L\rangle$ with 
$\alpha_i =(\alpha_{i,1},\ldots, \alpha_{i,n+1}) \in B_{m_i}$ 
here is different from 
the one in (\ref{BV}).} 
\begin{align}\label{tel}
T|\beta_1,\ldots, \beta_L\rangle 
= \sum_{\alpha_i \in B_{m_i}} T_{\beta_1,\ldots, \beta_L}^{\alpha_1,\ldots, \alpha_L}
|\alpha_1,\ldots, \alpha_L\rangle 
\in V_{m_1} \otimes \cdots \otimes V_{m_L}.
\end{align}
The matrix element is depicted as the concatenation of (\ref{vertex}) as 
\begin{equation}\label{tdiag}
\begin{picture}(250,50)(10,-25)
\put(-12,0){$T_{\beta_1,\ldots, \beta_L}^{\alpha_1,\ldots, \alpha_L}=
{\displaystyle \sum_{\gamma_1,\ldots, \gamma_L \in B_l}}$}

\put(100,0){
\put(0,0){\vector(1,0){24}}
\put(12,-12){\vector(0,1){24}}
\put(-12,-2){$\gamma_L$}\put(28,-2){$\gamma_1$}
\put(9,-22){$\beta_1$}\put(9,16){$\alpha_1$}}

\put(140,0){
\put(0,0){\vector(1,0){24}}
\put(12,-12){\vector(0,1){24}}
\put(27,-2){$\gamma_2$}
\put(9,-22){$\beta_2$}\put(9,16){$\alpha_2$}}

\put(182,-3){$\cdots$}

\put(220,0){
\put(0,0){\vector(1,0){24}}
\put(12,-12){\vector(0,1){24}}
\put(-24,-2){$\gamma_{L-1}$}\put(27,-2){$\gamma_L$.}
\put(9,-22){$\beta_L$}\put(9,16){$\alpha_L$}}
\end{picture}
\end{equation}
By the construction the $T$ satisfies the weight conservation:
\begin{align}\label{wc}
T_{\beta_1,\ldots, \beta_L}^{\alpha_1,\ldots, \alpha_L} = 0
\;\;\text{unless}\;\;
\alpha_1+\cdots +\alpha_L = 
\beta_1+\cdots + \beta_L \in \Z_{\ge 0}^{n+1}.
\end{align}

Next we proceed to the transfer matrix associated with the 
specialized $S$ matrix $\mathscr{S}(\lambda,\mu)$ in (\ref{smdef}):
\begin{align}\label{ioa}
\mathscr{T}(\lambda|\mu_1,\ldots, \mu_L) = 
\mathrm{Tr}_{W}\left(
\mathscr{S}_{0,L}(\lambda,\mu_L)\cdots \mathscr{S}_{0,1}(\lambda,\mu_1)
\right)
\in \mathrm{End}(W^{\otimes L}),
\end{align}
where the notations are similar to (\ref{tdef}).
Its matrix element 
$\mathscr{T}_{\beta_1,\ldots, \beta_L}^{\alpha_1,\ldots, \alpha_L} $
is again given by (\ref{tdiag}) if 
the $i$-th vertex from the left is regarded as 
$\mathscr{S}(\lambda,\mu_i)^{\gamma_i, \alpha_i}_{\gamma_{i-1},\beta_i}$ 
in (\ref{lin}) and 
$\alpha_i$'s and the sum over $\gamma_i$'s are taken from $\Z_{\ge 0}^n$.
Since the summand vanishes unless $\gamma_i\le \beta_i$ for all $i$,
the sum (\ref{tdiag}) for $\gamma_i \in\Z_{\ge 0}^n$ is finite
and $\mathscr{T}(\lambda|\mu_1,\ldots, \mu_L)$ is well-defined.
We have the commutativity
\begin{align*}
[ \mathscr{T}(\lambda|\mu_1,\ldots, \mu_L), 
\mathscr{T}(\lambda'|\mu_1,\ldots, \mu_L) ]=0
\end{align*}
and the weight conservation analogous to (\ref{wc}).

\subsection{Discrete time Markov chain with particle number constraint}\label{ss:dmc1}
Let us extract discrete time Markov processes by specializing 
the transfer matrix (\ref{tdef}).
First we consider a system governed by the evolution equation 
\begin{align}\label{PP}
|P(t+1)\rangle = T(l,z|{\textstyle {m_1,\ldots, m_L \atop w_1,\ldots, w_L})}
|P(t)\rangle \in V_{m_1} \otimes \cdots \otimes V_{m_L}.
\end{align}
It admits an interpretation as the master equation of 
a Markov process with the discrete time variable $t$ if 
$T=T(l,z|{\textstyle {m_1,\ldots, m_L \atop w_1,\ldots, w_L})}$ 
satisfies 
\begin{enumerate}
\item  Non-negativity; all the elements (\ref{tdiag}) belong to $\R_{\ge 0}$,

\item Sum-to-unity property; $\sum_{\alpha_1,\ldots, \alpha_L}
T_{\beta_1,\ldots, \beta_L}^{\alpha_1,\ldots, \alpha_L} = 1$ for any
$(\beta_1,\ldots, \beta_L) \in B_{m_1} \times \cdots \times B_{m_L}$.
\end{enumerate}
The latter represents the total probability conservation.
In order to satisfy them, we introduce the specialization
\begin{align}\label{spee}
T(l | m_1,\ldots, m_L) 
:= T(l,q^l | {\textstyle {m_1,\ldots, \,m_L \atop q^{m_1},\ldots,q^{m_L}})}\qquad
\text{for }\; l \in \Z_{\ge 0},
\end{align}
which still forms a commuting family
$[T(l | m_1,\ldots, m_L) , T(l' | m_1,\ldots, m_L) ]=0$
as a consequence of (\ref{ask}).
Now we see that (\ref{spee}) satisfies the above conditions (i) and (ii) 
provided that
$l \le \min\{m_1,\ldots, m_L\}$ and $q \in \R_{>0}$.
In fact, $l \le \min\{m_1,\ldots, m_L\}$ implies that  
all the relevant $S$ matrices in (\ref{tdef}) are reduced to the form 
(\ref{ikt}) from which (i) is obvious.
To confirm (ii),  evaluate 
$\sum_{\alpha_1,\ldots, \alpha_L}
T_{\beta_1,\ldots, \beta_L}^{\alpha_1,\ldots, \alpha_L}$
by substituting (\ref{ycn}) into (\ref{tdef}) or (\ref{tdiag}) as
\begin{align*}
&\sum_{\alpha_i \in B_{m_i}}
\sum_{\gamma_1,\ldots, \gamma_L \in B_l}
\delta^{\alpha_1+\gamma_1}_{\gamma_L+\beta_1}
\Phi_{q^2}(\bar{\gamma}_1|\bar{\beta}_1;q^{-2l},q^{-2m_1}) 
\cdots
\delta^{\alpha_L+\gamma_L}_{\gamma_{L-1}+\beta_L}
\Phi_{q^2}(\bar{\gamma}_L|\bar{\beta}_L;q^{-2l},q^{-2m_L}) 
\\
&= \sum_{\gamma_1,\ldots, \gamma_L \in B_l}
\theta(\gamma_1 \le \gamma_L+\beta_1)
\Phi_{q^2}(\bar{\gamma}_1|\bar{\beta}_1;q^{-2l},q^{-2m_1}) 
\cdots  
\theta(\gamma_L \le \gamma_{L-1}+\beta_L)
\Phi_{q^2}(\bar{\gamma}_L|\bar{\beta}_L; q^{-2l},q^{-2m_L}).
\end{align*}
One may remove $\theta(\gamma_i \le \gamma_{i-1}+\beta_i)$ for any $i$ 
since $\Phi_{q^2}(\bar{\gamma}_i|\bar{\beta}_i;q^{-2l},q^{-2m_i})=0$
unless $\bar{\gamma}_i \le \bar{\beta}_i$.
Note further that
$\gamma_i=(\gamma_{i,1},\ldots, \gamma_{i,n+1}) \in B_l$ is
in one-to-one correspondence with 
$\bar{\gamma}_i=(\gamma_{i,1},\ldots, \gamma_{i,n}) \in \Z_{\ge 0}^n$ such that 
$|\bar{\gamma}_i| \le l$.
Therefore the sum over $\gamma_i \in B_l$ may be replaced by 
$\bar{\gamma}_i \in \Z_{\ge 0}^n$ such that $|\bar{\gamma}_i| \le l$.
Then the above sum is evaluated by applying 
$(\ref{pwa})|_{q\rightarrow q^2}$, yielding $1$.

In this way we obtain a commuting family of 
evolution systems associated with (\ref{spee})
among which the cases $l\le \min\{m_1,\ldots, m_L\}$ 
can be regarded as discrete time Markov processes.

The diagram (\ref{tdiag}) is naturally interpreted 
in terms of $n$ species of particles obeying stochastic 
dynamics on the one-dimensional lattice.
It is supplemented with an extra lane (auxiliary space)
which particles get on or get off when they leave or arrive at a site.
The local situation at the $i$-th site from the left with 
$\beta_i=(\beta_{i,1},\ldots, \beta_{i,n+1}) \in B_{m_i}$ and 
$\gamma_i = (\gamma_{i,1},\ldots, \gamma_{i,n+1}) \in B_l$ is 
depicted as follows.
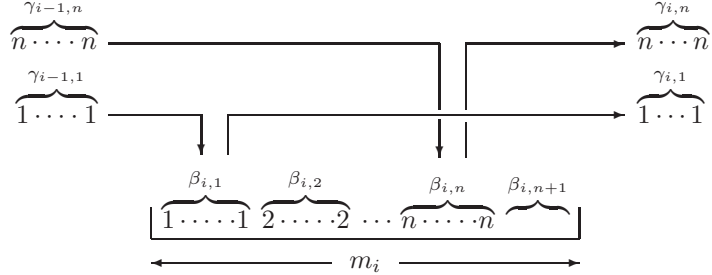
\begin{equation*}
\begin{picture}(300,100)(-80,-12)

% n in
\put(-67,73){$\overbrace{n \cdots \cdot  n}^{\gamma_{i-1,n}}$}
\put(-30,75){\line(1,0){125}}
\put(95,75){\line(0,-1){25}}
\put(95,46){\vector(0,-1){14}}

%n out
\put(105,32){\line(0,1){13}}
\put(105,50){\line(0,1){25}}
\put(105,75){\vector(1,0){60}}
\put(170,73){$\overbrace{n \cdots  n}^{\gamma_{i,n}}$}

\put(0,1){
%1 in
\put(-65,44){$\overbrace{1 \cdots \cdot 1}^{\gamma_{i-1,1}}$}
\put(-30,47){\line(1,0){35}}\put(5,47){\vector(0,-1){15}}
%1 out 
\put(15,32){\line(0,1){15}}\put(15,47){\vector(1,0){150}}
\put(170,44){$\overbrace{1\cdots1}^{\gamma_{i,1}}$}
}

\put(-5,0){
\put(-5,5){$\overbrace{1 \cdots \cdot \cdot1}^{\beta_{i,1}} \;
\overbrace{2 \cdots \cdot \cdot 2}^{\beta_{i,2}} \;
\cdots \,
\overbrace{n \cdots \cdot \cdot n}^{\beta_{i,n}} \;
\overbrace{\phantom{nnnn}}^{\beta_{i,n+1}}$}

\put(7,0){\put(-16,1){\line(1,0){162}}\put(-16,1){\line(0,1){12}}
\put(146,1){\line(0,1){12}}}
\put(6,0){\put(53,-8){\vector(-1,0){68}}
\put(60,-10){$m_i$}
\put(79,-8){\vector(1,0){68}}}}

\end{picture}
\end{equation*}
The site $i$ can accommodate up to $m_i$ particles.
The $\beta_{i,a}$ is the number of particles of species $a$ 
for $a \in [1,n]$ and the vacancy for $a=n+1$.
Among the $\beta_{i,a}$ particles of species $a$,
$\gamma_{i,a}\,(\le \beta_{i,a})$ of them are moving out to the right 
while $\gamma_{i-1,a}$ are moving in from the left.
The former event contributes the factor
$\Phi_{q^2}(\bar{\gamma}_i|\bar{\beta}_i; q^{-2l}, q^{-2m_i})$
to the total rate.
The number of particles on the extra lane 
is at most $l$ at every border of the adjacent sites.
Such a dynamics is closely parallel with its deterministic counterpart,
an integrable cellular automaton known as 
{\em box-ball system} with capacity-$l$ carrier and 
capacity-$m_i$ box at site $i$.
See \cite{IKT} and references therein.
 
 \subsection{Discrete time Markov chain 
without particle number constraint}\label{ss:dmc2}
Let us proceed to the system associated with the transfer matrix 
(\ref{ioa}) whose evolution is governed by
\begin{align}\label{dmt}
|P(t+1)\rangle = \mathscr{T}(\lambda|\mu_1,\ldots, \mu_L)
|P(t)\rangle \in W^{\otimes L}.
\end{align}
Although this is an equation in an infinite-dimensional vector space,
it actually splits into finite-dimensional subspaces specified by the 
particle content as 
$\mathscr{T}(\lambda|\mu_1,\ldots, \mu_L)$ preserves the weight.
One can satisfy the axioms (i) and (ii) for the discrete time Markov process 
stated after (\ref{PP}).
In fact, the non-negativity (i)  holds if 
$\Phi_q(\gamma|\beta; \lambda,\mu_i)\ge 0$
for all $i \in [1,L]$.
This is achieved 
by taking $0 < \mu^{\epsilon}_i < \lambda ^{\epsilon} < 1, q^{\epsilon}<1$
in the either alternative $\epsilon=\pm 1$.
The sum-to-unity condition (ii) 
$\sum_{\alpha_1,\ldots, \alpha_L}
\mathscr{T}_{\beta_1,\ldots, \beta_L}^{\alpha_1,\ldots, \alpha_L} = 1$
is valid thanks to (\ref{syk}).
The resulting stochastic dynamical system is parallel with the previous one 
associated with $T(l|m_1,\ldots, m_L)$
under the formal correspondence 
$\lambda = q^{-l}, \mu_i = q^{-m_i}$.
See (\ref{taio}).
The most notable difference, however, 
is that for the generic $\lambda, \mu_1, \ldots, \mu_L$ in the present setting,
there is no upper bound 
on the number of particles occupying a site $i$ nor those hopping 
from $i$ to $i+1$ ($i \mod L$).
It is described by the $n$-component arrays 
$\beta_i, \gamma_i \in \Z_{\ge 0}^n$
with the local transition rate factor 
$\Phi_q(\gamma_i| \beta_i; \lambda,\mu_i)$ (\ref{Pdef}).
When $n=1$ and $\mu_1=\cdots = \mu_L$,
such a system was introduced originally in \cite{P}.
As discussed therein, 
one can control the number of hopping particles in various ways
by specializing $\lambda, \mu$.

\subsection{Continuous time Markov chains}\label{ss:ctm}

Let us consider the discrete time Markov process described by (\ref{dmt}) with 
the homogeneous choice of the parameters $\mu_1= \cdots = \mu_L=\mu$.
We write the relevant Markov transfer matrix (\ref{ioa}) as
\begin{align}\label{tau}
\tau(\lambda|\mu):=\mathscr{T}(\lambda|\mu, \ldots, \mu),
\end{align}
which forms a commuting family
$[\tau(\lambda|\mu), \tau(\lambda'|\mu)]=0$.
The matrix elements of (\ref{tau}) are sums of products of 
$\Phi_q(\gamma|\beta; \lambda,\mu) $ (\ref{lin})  where 
the arrays like $\beta, \gamma$ are $n$-component ones.
The discrete time Markov process (\ref{dmt}) can be converted to 
a continuous time process by
taking the either limit
$\lambda \rightarrow 1$ 
or $\lambda \rightarrow \mu$ as we shall explain below.

First we treat the case  $\lambda \rightarrow 1$.
The relevant limiting formulas are as follows\footnote{The small 
expansion parameter $\Delta$ here
should not be confused with the coproduct in (\ref{Delta}).}:
\begin{align}
\Phi_q(\gamma|\beta; 1+\Delta,\mu) 
&= \Phi_q(\gamma|\beta; 1, \mu) 
+ {\Delta}\,\Phi'_q(\gamma|\beta; 1,\mu)
+O(\Delta^2),\nonumber\\
\Phi_q(\gamma|\beta; 1, \mu)
&= \delta_{\gamma,0}, \qquad
\mathscr{S}(1,\mu)_{\alpha, \beta}^{\gamma,\delta} 
= \delta_{\alpha+\beta}^\delta \delta_{\gamma,0},
\label{g0}\\
\Phi'_q(\gamma|\beta; 1,\mu) &:=
\left.\frac{\partial \Phi_q(\gamma|\beta; \lambda, \mu)}
{\partial \lambda}\right|_{\lambda=1}=
\begin{cases}
-q^{\xi}\mu^{|\gamma|}\frac{(q)_{|\gamma|-1}}
{(\mu q^{|\beta|-|\gamma|};q)_{|\gamma|}}
\prod_{i=1}^{n}\binom{\beta_i}{\gamma_i}_{\! q}
& \text{if }\,  |\gamma|>0,\label{yum}\\
\sum_{i=0}^{|\beta|-1}\frac{\mu q^i}{1-\mu q^i}& \text{if } \, |\gamma|=0,
\end{cases}
\end{align}
where $0=(0,\ldots, 0)\in \Z^n$ 
and $\xi$ is given by (\ref{xi}).
By the definition $\Phi'_q(\gamma|\beta; 1,\mu) =0$ unless $\gamma \le \beta$.
From (\ref{g0}), the element of $\tau(\lambda|\mu)$ 
(defined and depicted similarly to (\ref{tel}) and (\ref{tdiag})) is expanded as
\begin{equation}\label{texp}
\begin{picture}(400,115)(-380,-85)

\put(-377,-2){$\tau
(\lambda=1+\Delta| \mu)^{\alpha_1,\ldots, \alpha_L}_{\beta_1,\ldots, \beta_L}=$} 
\put(-185,0){
\put(-60,0){
\put(0,0){\vector(1,0){24}}
\put(12,-12){\vector(0,1){24}}
\put(-10,-3){$0$}
\put(27,-3){$0 \;\cdots $}
\put(9,-22){$\beta_1$}\put(9,16){$\alpha_1$}}
\put(0,0){\vector(1,0){24}}
\put(12,-12){\vector(0,1){24}}
\put(-8,-3){$0$}\put(27,-3){$0$}
\put(9,-22){$\beta_i$}\put(9,16){$\alpha_i$}
\put(36,0){
\put(0,0){\vector(1,0){24}}
\put(12,-12){\vector(0,1){24}}
%\put(-10,-2){$\alpha$}
\put(27,-3){$0\;\cdots$}
\put(9,-22){$\beta_{i+1}$}\put(9,16){$\alpha_{i+1}$}}
\put(99,0){
\put(0,0){\vector(1,0){24}}
\put(12,-12){\vector(0,1){24}}
\put(-10,-3){$0$}
\put(27,-3){$0$}
\put(9,-22){$\beta_L$}\put(9,16){$\alpha_L$}}
}

\put(-185,-55){
\put(-132,-2){$+ \;\Delta {\displaystyle \sum_{i \in \Z_L}
\sum_{\gamma_i \in \Z_{\ge 0}^n}}$}

\put(-60,0){
\put(0,0){\vector(1,0){24}}
\put(12,-12){\vector(0,1){24}}
\put(-10,-3){$0$}
\put(27,-3){$0 \;\cdots $}
\put(9,-22){$\beta_1$}\put(9,16){$\alpha_1$}}
\put(0,0){\vector(1,0){24}}
\put(12,-12){\vector(0,1){24}}
\put(9.6,-2.4){$\circ$}
\put(-8,-3){$0$}\put(27,-2){$\gamma_i$}
\put(9,-22){$\beta_i$}\put(9,16){$\alpha_i$}
\put(36,0){
\put(0,0){\vector(1,0){24}}
\put(12,-12){\vector(0,1){24}}
%\put(-10,-2){$\alpha$}
\put(27,-3){$0\;\cdots$}
\put(9,-22){$\beta_{i+1}$}\put(9,16){$\alpha_{i+1}$}}
\put(99,0){
\put(0,0){\vector(1,0){24}}
\put(12,-12){\vector(0,1){24}}
\put(-10,-3){$0$}
\put(27,-3){$0\;\; \;+ \;O(\Delta^2).$}
\put(9,-22){$\beta_L$}\put(9,16){$\alpha_L$}}
}
\end{picture}
\end{equation}
The vertices here denote 
$\mathscr{S}(\lambda\!=\!1,\mu)_{\alpha, \beta}^{\gamma,\delta}$.
The first term leads to $\tau(1|\mu) = \mathrm{id}_{W^{\otimes L}}$
owing to $\mathscr{S}(1,\mu)^{0, \alpha_i}_{0, \beta_i} 
= \delta^{\alpha_i}_{\beta_i}$ by (\ref{g0}).
In the second term, the mark 
$\circ$ signifies the unique vertex corresponding to the 
derivative (\ref{yum}).
Its ``vertex weight" is equal to 
$\frac{\partial}{\partial \lambda}
\mathscr{S}(\lambda,\mu)^{\gamma_i, \alpha_i}_{0,\beta_i}
\vert_{\lambda=1}
= \delta^{\alpha_i+\gamma_i}_{\beta_i}
\Phi'_q(\gamma_i|\beta_i;1,\mu)$
calculated in (\ref{yum}).
Introduce the local (adjacent) transition rate 
$w\bigl((\alpha,\beta) \rightarrow 
(\rho,\sigma)\bigr)$ by
\begin{equation}\label{wdef}
\begin{picture}(400,50)(-150,-25)

\put(-143,-2){$-\epsilon \mu w\bigl((\alpha,\beta) \rightarrow 
(\rho,\sigma)\bigr) = 
{\displaystyle \sum_{\gamma \in \Z_{\ge 0}^n}}$}

\put(0,0){\vector(1,0){24}}
\put(12,-12){\vector(0,1){24}}
\put(9.5,-2.5){$\circ$}
\put(-8,-3){$0$}\put(27,-2){$\gamma$}
\put(9,-22){$\alpha$}\put(9,16){$\rho$}
\put(36,0){
\put(0,0){\vector(1,0){24}}
\put(12,-12){\vector(0,1){24}}
%\put(-10,-2){$\alpha$}
\put(27,-3){$0$}
\put(9,-22){$\beta$}\put(9,16){$\sigma$}}

\put(98,0){
\put(-22,-2){$=$}
\put(0,0){\vector(1,0){24}}
\put(12,-12){\vector(0,1){24}}
\put(9.5,-2.5){$\circ$}
\put(-9,-3){$0$}
\put(27,-3){$\sigma-\beta \;
= \; \delta_{\alpha+\beta}^{\rho+\sigma}\,
\Phi'_q(\alpha-\rho|\alpha; 1, \mu).$}
\put(9,-22){$\alpha$}\put(9,16){$\rho$}}

\end{picture}
\end{equation}
Here $\epsilon=\pm 1$ has been inserted to distinguish the 
two regimes of the model as we shall explain below.
The extra minus sign is included in view of that in (\ref{yum}).
The rate satisfies 
\begin{align}\label{w0}
-\epsilon\mu \sum_{\rho, \sigma \in \Z_{\ge 0}^n}
w\bigl((\alpha,\beta) \rightarrow 
(\rho,\sigma)\bigr)=
\sum_{\rho \in \Z_{\ge 0}^n}
\theta(\rho \le \alpha+\beta)\Phi'_q(\alpha-\rho|\alpha; 1, \mu)
= \sum_{\gamma \in \Z_{\ge 0}^n}\Phi'_q(\gamma|\alpha; 1, \mu)
=0,
\end{align}
where the last equality follows by differentiating (\ref{syk}) 
with respect to $\lambda$ and setting $\lambda=1$ afterwards.

According to a general construction, we introduce the
matrix $h(\mu) \in \mathrm{End}(W \otimes W)$ by
\begin{equation}\label{hact}
\begin{split}
h(\mu) |\alpha, \beta\rangle 
&= \sum_{\rho, \sigma \in \Z_{\ge 0}^n}
h(\mu)_{\alpha,\beta}^{\rho,\sigma} 
 |\rho, \sigma\rangle,\\
h(\mu)_{\alpha,\beta}^{\rho,\sigma}
&= w\bigl((\alpha,\beta) \!\rightarrow\! (\rho,\sigma)\bigr) - 
\delta_{\alpha}^{\rho}\delta_{\beta}^{\sigma}\!
\sum_{\rho', \sigma' \in \Z_{\ge 0}^n}\!\!\!
w\bigl((\alpha,\beta)\!\rightarrow\!(\rho',\sigma')\bigr)
= -\epsilon \mu^{-1}\,\delta_{\alpha+\beta}^{\rho+\sigma}\,
\Phi'_q(\alpha-\rho|\alpha; 1, \mu),
\end{split}
\end{equation}
The last equality is due to (\ref{wdef}) and (\ref{w0}).
For an interpretation as a local Markov matrix in a continuous time process,
the $h(\mu)$ should satisfy 

\vspace{0.1cm}
(i)' Non-negativity; $h(\mu)_{\alpha,\beta}^{\rho,\sigma} \ge 0$ for
$(\rho,\sigma) \neq (\alpha,\beta)$,
 
\hspace{-0.1cm}(ii)' Sum-to-zero property; 
$\sum_{\rho, \sigma}h(\mu)_{\alpha,\beta}^{\rho,\sigma}=0$,

\vspace{0.1cm}\noindent
which are analogue of (i) and (ii) mentioned after (\ref{PP})
for the discrete time case. 
We see that (i)' holds if 
$0< q^{\epsilon}, \mu^{\epsilon} <1$ from the explicit formula (\ref{yum}).
The property (ii)' is obvious by the construction. 

Now the expansion (\ref{texp}) is expressed as 
\begin{align}\label{te2}
\tau(\lambda=1+\Delta|\mu) = 
\mathrm{id}_{W^{\otimes L}} -\epsilon \mu\Delta H+O(\Delta^2),
\quad H= \sum_{i \in \Z_L}h(\mu)_{i,i+1},
\end{align}
where $h(\mu)_{i,i+1}$ is the local Markov matrix 
(\ref{hact}) acting on the $i$-th and the $(i+1)$-th sites.
Picking the $O(\Delta)$ terms  
in the time-scaled master equation
$|P(t-\epsilon\mu\Delta)\rangle = \tau(\lambda=1+\Delta|\mu)|P(t)\rangle$
and applying (\ref{te2}), 
we obtain the continuous time master equation:
\begin{align}\label{ctmp1}
\frac{d}{dt}|P(t)\rangle =H|P(t)\rangle.
\end{align}
The local Markov matrix (\ref{hact}) 
acts on the neighboring sites as follows:
\begin{align}\label{hee}
h(\mu) |\alpha, \beta\rangle 
&= -\epsilon\mu^{-1}\sum_{\gamma \in \Z_{\ge 0}^n}
\Phi'_q(\gamma|\alpha; 1,\mu)|\alpha-\gamma, \beta+\gamma\rangle,
\qquad\qquad\qquad
\begin{picture}(60,25)(0,13)
\drawline(0,5)(0,0)
\drawline(0,0)(50,0)\drawline(50,0)(50,5)
\drawline(25,0)(25,5)
\put(-1,3){
\drawline(12,12)(12,22)\drawline(12,22)(37,22)\put(37,22){\vector(0,-1){10}}
}
\put(9,4){$\alpha$}\put(33.5,4){$\beta$}
\put(21,31){$\gamma$}
\end{picture}
\end{align}
It defines a stochastic dynamics among $n$ species of particles 
on a one-dimensional lattice.
There is no constraint on the particles for sharing the same site.
They jump only to the right adjacent site without any constraint 
on their occupancy at the destination site either.
If there are $\alpha_a$ particles of species $a$ at the departure site and  
$\gamma_a (\le \alpha_a)$ of them are moving out together,
the associated transition rate is
\begin{align}\label{rate}
-\epsilon\mu^{-1}\Phi'_q(\gamma|\alpha;1, \mu)
=\epsilon 
\frac{q^{\sum_{1\le i<j \le n}(\alpha_i-\gamma_i)\gamma_j}
\mu^{\gamma_1+\cdots + \gamma_n-1}(q)_{\gamma_1+\cdots + \gamma_n-1}}
{(\mu q^{\alpha_1+\cdots +\alpha_n-\gamma_1-\cdots -\gamma_n};
q)_{\gamma_1+\cdots + \gamma_n}}
\prod_{i=1}^n
\binom{\alpha_i}{\gamma_i}_{\!q}
\end{align}
for a nontrivial case, i.e. if 
$\gamma_1+\cdots + \gamma_n\ge 1$.
For $\epsilon=\pm 1$ and the parameters $q$ and $\mu$ such that 
$0 \le q^\epsilon, \mu^\epsilon <1$, 
it defines a new $n$-species TAZRP.

When $\mu=0$, the local transition rate (\ref{rate}) is nonzero only for   
$|\gamma|=1$ or $\gamma=0$ (no transition).
In the former case it simplifies to 
\begin{align*}
\frac{1-q^{\alpha_b}}{1-q}q^{\sum_{j=1}^{b-1}\alpha_j}
\end{align*}
if $\epsilon=1$ and 
the species of the single particle to hop is $b$, i.e. $\gamma_b=1$. 
It coincides with the rate in \cite[p1]{T} 
upon reversing the labeling of the species.
The single species case $n=1$ further goes back to 
the $q$-boson model \cite{SW}.
When $n=1$ and $\epsilon=1$, the formula (\ref{rate}) for general $\mu$ 
is proportional to the rate given in \cite[p2]{T0} 
under the identification $\mu=s/(1-q+s)$.

Next, we proceed to another continuous time Markov chain which 
arises from (\ref{tau}) at $\lambda=\mu$.
As it turns out, this is closer to the usual derivation of spin chain Hamiltonians  
(cf. \cite[Chap. 10.14]{Bax})
than $\lambda=1$.
The relevant limiting formulas read
\begin{align}
\Phi_q(\gamma|\beta; \mu+\Delta,\mu) 
&= \Phi_q(\gamma|\beta; \mu, \mu) 
+ \Delta\,\Phi'_q(\gamma|\beta; \mu,\mu)
+O(\Delta^2),\nonumber\\
\Phi_q(\gamma|\beta; \mu, \mu)
&= \delta_{\gamma,\beta}, \qquad
\mathscr{S}(\mu,\mu)_{\alpha, \beta}^{\gamma,\delta} 
= \delta_{\alpha}^\delta \delta_\beta^\gamma,
\label{gk}\\
\Phi'_q(\gamma|\beta; \mu,\mu) &:=
\left.\frac{\partial \Phi_q(\gamma|\beta; \lambda, \mu)}
{\partial \lambda}\right|_{\lambda=\mu}=
\begin{cases}
\mu^{-1}q^{\xi}\frac{(q)_{|\beta|-|\gamma|-1}}
{(\mu q^{|\gamma|};q)_{|\beta|-|\gamma|}}
\prod_{i=1}^{n}\binom{\beta_i}{\gamma_i}_{\! q}
& \text{if }\,  |\beta|>|\gamma|,\label{yuk}\\
\mu^{-1}\sum_{i=0}^{|\beta|-1}\frac{-1}{1-\mu q^i}
& \text{if } \, |\beta|=|\gamma|,
\end{cases}
\end{align}
where $\xi$ is again given by (\ref{xi}).
The result (\ref{gk}) is depicted as a local shift:
\begin{equation}\label{suu}
\begin{picture}(200,45)(-90,-21)
\put(-68,-2){$\mathscr{S}(\mu,\mu)_{\alpha, \beta}^{\gamma,\delta}=$}
\put(0,0){\line(1,0){10}}\put(14,0){\vector(1,0){10}}
\put(12,-12){\line(0,1){10}}\put(12,2){\vector(0,1){10}}
\drawline(10,0)(12,2)\drawline(12,-2)(14,0)
\put(-10,-2){$\alpha$}\put(27,-2){$\gamma$}
\put(9,-22){$\beta$}\put(9,16){$\delta$}
\end{picture}
\end{equation}
Consequently the expansion of $\tau(\lambda|\mu)$ 
in the vicinity of $\lambda = \mu$ takes the form
\begin{equation}\label{texp2}
\begin{picture}(400,115)(-380,-85)

\put(-363,-2){$\tau
(\lambda=\mu+\Delta| \mu)^{\alpha_1,\ldots, \alpha_L}_{\beta_1,\ldots, \beta_L}=$} 
\put(-185,0){
\put(-60,0){
\put(0,0){\line(1,0){10}}\put(14,0){\vector(1,0){10}}
\put(12,-12){\line(0,1){10}}\put(12,2){\vector(0,1){10}}
\drawline(10,0)(12,2)\drawline(12,-2)(14,0)
\put(27,-3){$\;\cdots $}
\put(9,-22){$\beta_1$}\put(9,16){$\alpha_1$}}
\put(-15,0){
\put(0,0){\line(1,0){10}}\put(14,0){\line(1,0){10}}
\put(12,-12){\line(0,1){10}}\put(12,2){\vector(0,1){10}}
\drawline(10,0)(12,2)\drawline(12,-2)(14,0)
\put(9,-22){$\beta_i$}\put(9,16){$\alpha_i$}}
\put(9,0){
\put(0,0){\line(1,0){10}}\put(14,0){\line(1,0){10}}
\put(12,-12){\line(0,1){10}}\put(12,2){\vector(0,1){10}}
\drawline(10,0)(12,2)\drawline(12,-2)(14,0)
\put(-5,0){\put(9,-22){$\beta_{i+1}$}\put(9,16){$\alpha_{i+1}$}}}

\put(33,0){
\put(0,0){\line(1,0){10}}\put(14,0){\vector(1,0){10}}
\put(12,-12){\line(0,1){10}}\put(12,2){\vector(0,1){10}}
\drawline(10,0)(12,2)\drawline(12,-2)(14,0)
\put(30,-3){$\;\cdots$}
\put(-3,0){\put(9,-22){$\beta_{i+2}$}\put(9,16){$\alpha_{i+2}$}}}
\put(83,0){
\put(0,0){\line(1,0){10}}\put(14,0){\vector(1,0){10}}
\put(12,-12){\line(0,1){10}}\put(12,2){\vector(0,1){10}}
\drawline(10,0)(12,2)\drawline(12,-2)(14,0)
\put(9,-22){$\beta_L$}\put(9,16){$\alpha_L$}}
}

\put(-185,-55){
\put(-100,-2){$+ \;\Delta {\displaystyle \sum_{i \in \Z_L}}$}

\put(-60,0){
\put(0,0){\line(1,0){10}}\put(14,0){\vector(1,0){10}}
\put(12,-12){\line(0,1){10}}\put(12,2){\vector(0,1){10}}
\drawline(10,0)(12,2)\drawline(12,-2)(14,0)
\put(27,-3){$\;\cdots $}
\put(9,-22){$\beta_1$}\put(9,16){$\alpha_1$}}
\put(-15,0){
\put(0,0){\line(1,0){10}}\put(14,0){\line(1,0){10}}
\put(12,-12){\line(0,1){10}}\put(12,2){\vector(0,1){10}}
\drawline(10,0)(12,2)\drawline(12,-2)(14,0)
\put(9,-22){$\beta_i$}\put(9,16){$\alpha_i$}}
\put(9,0){
\put(0,0){\line(1,0){24}}
\put(12,-12){\vector(0,1){24}}
\put(9.5,-2.5){$\diamond$}
\put(-5,0){\put(9,-22){$\beta_{i+1}$}\put(9,16){$\alpha_{i+1}$}}}

\put(33,0){
\put(0,0){\line(1,0){10}}\put(14,0){\vector(1,0){10}}
\put(12,-12){\line(0,1){10}}\put(12,2){\vector(0,1){10}}
\drawline(10,0)(12,2)\drawline(12,-2)(14,0)
\put(30,-3){$\;\cdots$}
\put(-3,0){\put(9,-22){$\beta_{i+2}$}\put(9,16){$\alpha_{i+2}$}}}
\put(83,0){
\put(0,0){\line(1,0){10}}\put(14,0){\vector(1,0){10}}
\put(12,-12){\line(0,1){10}}\put(12,2){\vector(0,1){10}}
\drawline(10,0)(12,2)\drawline(12,-2)(14,0)
\put(9,-22){$\beta_L$}\put(9,16){$\alpha_L$}}
\put(120,-2){$ + \; O(\Delta^2)$.}
}
\end{picture}
\end{equation}
All the matrices appearing here as coefficients of $\Delta^k\, (k=0,1,2,\ldots)$
commute with each other.
The first term $\tau(\mu|\mu)$  gives the $\Z_L$-cyclic shift operator of the chain.
In the second term, the vertex marked with $\diamond$ signifies 
$\frac{\partial}{\partial \lambda}
\mathscr{S}(\lambda,\mu)_{\beta_i, \beta_{i+1}}^{\alpha_{i+2}, \alpha_{i+1}}
\vert_{\lambda=\mu} = \delta_{\beta_i+\beta_{i+1}}^{\alpha_{i+1}+\alpha_{i+2}}
\Phi'_q(\alpha_{i+2}|\beta_{i+1};\mu,\mu)$ 
calculated in (\ref{yuk}).

Introduce the matrix 
$\hat{h}(\mu) \in \mathrm{End}(W \otimes W)$ by
$\hat{h}(\mu)|\alpha, \beta\rangle = 
\sum_{\gamma, \delta} \hat{h}(\mu)_{\alpha, \beta}^{\gamma, \delta}
|\gamma, \delta\rangle$ with the elements 
\begin{equation*}
\begin{picture}(200,45)(-90,-21)
\put(-80,-2){$\epsilon \mu^{-1}
\hat{h}(\mu)_{\alpha, \beta}^{\gamma, \delta}\, = $}
\put(0,0){\vector(1,0){24}}\put(12,-12){\vector(0,1){24}} \put(9.5,-2.5){$\diamond$}
\put(-10,-2){$\alpha$}
\put(9,-22){$\beta$}
\put(8.5,16.6){$\gamma$}
\put(27,-2){$\delta \;\, = \,
\delta_{\alpha+\beta}^{\gamma+ \delta}\,
\Phi'_q(\delta | \beta; \mu,\mu)$,}
\end{picture}
\end{equation*}
where $\epsilon= \pm 1$ is inserted again to label the two regimes of the model. 
We remark that the positions of $\gamma$ and $\delta$ in this diagram 
have been interchanged from those in (\ref{suu}).
From (\ref{yuk})  we see that 
$\hat{h}(\mu)_{\alpha, \beta}^{\gamma, \delta} \ge 0$ for 
$(\alpha, \beta) \neq (\gamma, \delta)$ if $0 \le q^\epsilon, \mu^\epsilon <1$.
Moreover $\sum_{\gamma, \delta} \hat{h}(\mu)_{\alpha, \beta}^{\gamma, \delta} =0$
holds by the reason similar to the last equality in (\ref{w0}).
Thus $\hat{h}(\mu)$ can be interpreted as a local Markov matrix.
The expansion (\ref{texp2}) is neatly presented by switching to
the transfer matrix in the ``moving frame" 
$\hat{\tau}(\lambda|\mu) := \tau(\mu|\mu)^{-1}\tau(\lambda|\mu)$ as
\begin{align}\label{te3}
\hat{\tau}(\lambda=\mu+\Delta|\mu) &= 
\mathrm{id}_{W^{\otimes L}} + \epsilon \mu^{-1} \Delta \hat{H} + 
O(\Delta^2),
\quad
\hat{H} = \sum_{i\in \Z_L}\hat{h}(\mu)_{i,i+1},\\
\hat{h}(\mu) |\alpha, \beta\rangle 
&= \epsilon \mu \sum_{\gamma \in \Z_{\ge 0}^n}
\Phi'_q(\beta-\gamma|\beta; \mu,\mu)|\alpha+\gamma, \beta-\gamma\rangle,
\qquad\qquad\qquad
\begin{picture}(60,25)(0,13)
\drawline(0,5)(0,0)
\drawline(0,0)(50,0)\drawline(50,0)(50,5)
\drawline(25,0)(25,5)
\put(-1,3){
\put(12,22){\vector(0,-1){10}}\drawline(12,22)(37,22)
\put(37,22){\line(0,-1){10}}
}
\put(9,4){$\alpha$}\put(33.5,4){$\beta$}
\put(21,31){$\gamma$}
\end{picture}\label{hee2}
\end{align}
where the sum is finite because the summand is zero unless $\gamma \le \beta$.
From the time-scaled master equation
$|P(t+ \epsilon \mu^{-1}\Delta)\rangle 
= \hat{\tau}(\lambda = \mu+ \Delta | \mu) |P(t)\rangle$, we get the 
continuous time master equation
\begin{align}\label{ctmp2}
\frac{d}{dt}|P(t)\rangle =\hat{H}|P(t)\rangle.
\end{align}
The rate for the nontrivial transition $\gamma_1+\cdots + \gamma_n\ge 1$
is given by
\begin{align}\label{rate2}
\epsilon \mu \Phi'_q(\beta-\gamma|\beta; \mu,\mu)
= \epsilon 
\frac{q^{\sum_{1\le i<j \le n}\gamma_i(\beta_j-\gamma_j)}
(q)_{\gamma_1+\cdots + \gamma_n-1}}
{(\mu q^{\beta_1+\cdots +\beta_n-\gamma_1-\cdots -\gamma_n};
q)_{\gamma_1+\cdots + \gamma_n}}
\prod_{i=1}^n
\binom{\beta_i}{\gamma_i}_{\!q}
\end{align}
when $\gamma_a(\le \beta_a)$ among the 
$\beta_a$ particles of species $a$ in the departure site are hopping to the left.
For $\epsilon=\pm 1$, it defines another $n$-species TAZRP
depending on the parameters $q$ and $\mu$ such that 
$0 \le q^\epsilon, \mu^\epsilon <1$.

To summarize so far, we have extracted the continuous time Markov matrices 
$H$ in (\ref{te2}) and $\hat{H}$ in (\ref{te3}) 
from $\tau(\lambda|\mu)$ (\ref{tau})  by the prescription so called  
{\em Baxter's formula}:
\begin{align}\label{bax}
H = \left.-\epsilon \mu^{-1}
\frac{\partial \log 
\tau(\lambda | \mu)}{\partial \lambda}\right|_{\lambda=1},
\qquad
\hat{H} = \left. \epsilon \mu\,
\frac{\partial \log \tau(\lambda | \mu)}{\partial \lambda} \right|_{\lambda = \mu},
\end{align}
where the former may also be presented as
$H = -\epsilon \mu^{-1}
\frac{\partial}{\partial \lambda} 
\tau(\lambda | \mu)|_{\lambda=1}$
in view of $\tau(1|\mu) = \mathrm{id}_{W^{\otimes L}}$.
By the construction $[H,\hat{H}]=0$ holds.
The $H$ (resp.~$\hat{H}$) 
represents the $n$-species TAZRP in which 
particles hop to the right (resp.~left) with the local transition
rate (\ref{rate}) (resp.~(\ref{rate2})).
They admit two regimes $\epsilon = \pm 1$ in which 
the parameters $q$ and $\mu$ should be taken in the range
$0 \le q^\epsilon, \mu^\epsilon < 1$.

It turns out that the two models can be identified 
through a certain transformation.
To explain it, let us exhibit the regime/parameter dependence as 
$H(\epsilon, q, \mu)$ and $\hat{H}(\epsilon, q, \mu)$.
The key to the equivalence is the identity
\begin{align}\label{PhP}
\mu^{-1}\Phi'_{q^{-1}}(\gamma|\beta; 1, \mu^{-1}) = 
\Phi'_q(\beta-\gamma | \beta; \mu, \mu),
\end{align}
which can be directly checked from (\ref{yum}) and (\ref{yuk}).
Comparing (\ref{hee}) and (\ref{hee2}) by applying (\ref{PhP}),
one finds that the two Markov matrices are linked as
\begin{align}\label{HH}
\mu^{-1} H(-\epsilon, q^{-1}, \mu^{-1})
= \mathscr{P} \hat{H}(\epsilon,q,\mu) \mathscr{P}^{-1}.
\end{align}
Here $\mathscr{P} = \mathscr{P}^{-1} 
\in \mathrm{End}(W^{\otimes L})$ is the ``parity" operator 
reversing the sites as
$\mathscr{P}|\sigma_1,\ldots, \sigma_L\rangle = 
|\sigma_L, \ldots, \sigma_1\rangle$ which adjusts the directions of  
$\gamma$-arrows in (\ref{hee}) and (\ref{hee2}).
Thus studying either one of $H$ or $\hat{H}$ for the two regimes $\epsilon = \pm 1$
is equivalent to treating the two models concentrating on either one of the regimes.
It is intriguing that two members in the commuting family 
$\{\tau(\lambda|\mu)\}$ with respect to $\lambda$ 
are linked by the relation like (\ref{HH}).
We will explain the coincidence of the spectra implied by it 
also at the level of Bethe ansatz around (\ref{He}).

\begin{remark}\label{re:H}
For any $a,b \in \R_{\ge 0}$, the combination 
${\mathcal H}(a,b,\epsilon,q,\mu) =a H(\epsilon,q,\mu)+ b \hat{H}(\epsilon,q,\mu)$
satisfies 
${\mathcal H}(a,b,-\epsilon, q^{-1}, \mu^{-1})
= \mathscr{P}{\mathcal H}(\mu b, \mu a,\epsilon,q,\mu)\mathscr{P}^{-1}$
and possesses the spectrum obtained by superposing (\ref{He}) correspondingly. 
For $0 \le q^\epsilon, \mu^\epsilon < 1$, it defines a Markov matrix 
of the integrable asymmetric zero range process 
in which the particles can hop to the {\em both} directions.
\end{remark}

Let us include a comment on the model corresponding to $\hat{H}(1,q,0)$.
From (\ref{rate2}), the relevant transition rate is
\begin{align}\label{qKMO}
\lim_{\mu\to 0}\mu \Phi'_q(\beta-\gamma|\beta; \mu,\mu)
= \begin{cases}
{\displaystyle q^{\sum_{1 \le i<j \le n}\gamma_i(\beta_j-\gamma_j)}
(q)_{\gamma_1+\cdots + \gamma_n-1}
\prod_{i=1}^n\binom{\beta_i}{\gamma_i}_{\!q}}
& \text{if } |\gamma|\ge 1,\\
-(\beta_1+\cdots +\beta_n) & \text{if } \gamma=0.
\end{cases}
\end{align}
It defines a one-parameter family of 
integrable $n$-species TAZRP for $0 \le q<1$.
In particular at $q=0$, the local dynamics  
is frozen to the situation
$\sum_{1 \le i<j \le n}\gamma_i(\alpha_j-\gamma_j)=0$.
To digest this constraint, 
let $s$ be the minimum of the species of the particles 
that are jumping out. 
Namely, $s\in [1,n]$ is the smallest among those satisfying $\gamma_s>0$.
Then the above condition implies 
$\gamma_a=\alpha_a$ for all $a \in [s+1,n]$.
It means that all the particles with species larger than $s$
must also be jumping out simultaneously.
In other words, larger species particles 
always have the {\em priority} in the multiple particle jumps,
and all such events have an equal rate.
Such a stochastic dynamics exactly coincides with the  
$n$-species TAZRP in \cite{KMO}  with the 
homogeneous choice of the parameters 
$w_1=\cdots =w_n$ therein.
Thus (\ref{qKMO}) can be viewed as defining 
an integrable {\em $q$-melting} of it.

\begin{remark}\label{re:rbc}
Our particle interpretation here and the previous subsection is entirely
based on regarding the first $n$ components in 
the arrays $\alpha=(\alpha_1, \ldots, \alpha_{n+1})$ 
as the number of $n$ species of particles.
However it is a matter of option which components one regards so.
Changing them would lead to apparently different variety of 
stochastic dynamics of multispecies particle systems.
\end{remark}

\section{Bethe eigenvalues}\label{sec:ba}

\subsection{\mathversion{bold}Spectrum of 
$T(l,z|{\textstyle {m_1,\ldots, m_L \atop w_1,\ldots, w_L}})$}

Let $\Lambda(l,z|{\textstyle {m_1,\ldots, m_L \atop w_1,\ldots, w_L}})$
denote  the eigenvalues of the transfer matrix 
$T(l,z|{\textstyle {m_1,\ldots, m_L \atop w_1,\ldots, w_L}})$ 
(\ref{tdef}) where $l,m_i \in \Z_{\ge 1}$.
It is described by the Bethe ansatz.
See for example  \cite[Chap.7,8]{KNS} 
for a review and also \cite{FH} for a recent development.

We first illustrate the $U_q(A^{(1)}_1)$ case.
Consider the subspace of 
$\bigoplus \,\C |\alpha_1,\ldots, \alpha_L\rangle 
\in V_{m_1} \otimes \cdots \otimes V_{m_L}$
specified by the weight condition on the arrays
$\alpha_i=(\alpha_{i,1},\alpha_{i,2}) \in B_{m_i}$ as 
$(\sum_{i=1}^L\alpha_{i,1}, \sum_{i=1}^L\alpha_{i,2})
= (N_1, \sum_{i=1}^Lm_i-N_1)$.
The $T(l,z|{\textstyle {m_1,\ldots, m_L \atop w_1,\ldots, w_L}})$
is the transfer matrix of a higher spin vertex model 
whose auxiliary space is degree $l$ symmetric tensor representation $V_l$.
Its eigenvalues are given, for instance for $l=1,2$ by
\begin{align*}
\Lambda(1,z|{\textstyle {m_1,\ldots, m_L \atop w_1,\ldots, w_L}})
& = \frac{Q_1(qz)}{Q_1(q^{-1}z)}+ q^{2N_1}
\prod_{i=1}^L\left(\frac{q^{-m_i+1}w_i-z}{q^{m_i+1}w_i-z}\right)
\frac{Q_1(q^{-3}z)}{Q_1(q^{-1} z)},\\
\Lambda(2,z|{\textstyle {m_1,\ldots, m_L \atop w_1,\ldots, w_L}})
& = \frac{Q_1(q^2z)}{Q_1(q^{-2}z)}
+ q^{2N_1}
\prod_{i=1}^L\left(\frac{q^{-m_i+2}w_i-z}{q^{m_i+2}w_i-z}\right)
\frac{Q_1(q^2z)Q_1(q^{-4}z)}{Q_1(z)Q_1(q^{-2} z)}\\
&+ q^{4N_1}
\prod_{i=1}^L\left(\frac{q^{-m_i+2}w_i-z}{q^{m_i+2}w_i-z}
\frac{q^{-m_i}w_i-z}{q^{m_i}w_i-z}\right)
\frac{Q_1(q^{-4}z)}{Q_1(z)},
\end{align*} 
where 
$Q_1(z) = \prod_{k=1}^{N_1}(1-zu^{(1)}_k)$
is called the {\em Baxter $Q$ function} 
whose roots are determined by the Bethe equation:
\begin{align*}
-\prod_{i=1}^L\left(\frac{1-q^{-m_i}w_iu^{(1)}_j}{1-q^{m_i}w_iu^{(1)}_j}\right)
= q^{-2N_1}\frac{Q_1(q^2/u^{(1)}_j)}{Q_1(q^{-2}/u^{(1)}_j)}
= \prod_{k=1}^{N_1}
\frac{u^{(1)}_j-q^2u^{(1)}_k}{q^2u^{(1)}_j-u^{(1)}_k}.
\end{align*}
It is the generic pole-freeness condition of the eigenvalue formulas
despite the presence of zeroes in $Q_1(z)$.
The above $\Lambda(1,z|{\textstyle {m_1,\ldots, m_L \atop w_1,\ldots, w_L}})$
with $\forall m_i = 1$ and $w_1=\cdots = w_L$ 
corresponds to the homogeneous six-vertex model in 
\cite[eq.(8.9.13)]{Bax}.

Denote 
$T(l,z|{\textstyle {m_1,\ldots, m_L \atop w_1,\ldots, w_L}})$
simply by $T(l,z)$.
\
Then as the consequence of the fusion procedure, 
it is known to obey the {\em  T-system} (cf.  \cite{KNS}):
\begin{align*}
T(l,zq)T(l,zq^{-1}) = T(l+1,z)T(l-1,z) + 
q^{2lN_1}
\prod_{s=1}^l
\prod_{i=1}^L\left(\frac{q^{-m_i}w_i-q^{l-2s+1}z}
{q^{m_i}w_i-q^{l-2s+1}z}\right)
\mathrm{id},
\end{align*}
where $T(0,z)=\mathrm{id}$.
Solving the same recursion relation
for the eigenvalues starting from the initial condition $l=0,1$, one 
arrives at the formula for 
$\Lambda(l,z|{\textstyle {m_1,\ldots, m_L \atop w_1,\ldots, w_L}})$ with general $l$.
The result is presented neatly in terms of  
\begin{align*}
\framebox{1} _z=\frac{Q_1(qz)}{Q_1(q^{-1}z)},\quad
\framebox{2} _z=q^{2N_1}
\prod_{i=1}^L\left(\frac{q^{-m_i+1}w_i-z}{q^{m_i+1}w_i-z}\right)
\frac{Q_1(q^{-3}z)}{Q_1(q^{-1} z)}
\end{align*}
as the sum over one-row semistandard Young tableaux with entries from $\{1,2\}$.
For instance,
\begin{align*}
\Lambda(1,z|{\textstyle {m_1,\ldots, m_L \atop w_1,\ldots, w_L}})
= \framebox{1} _z + \framebox{2} _z,\quad
\Lambda(2,z|{\textstyle {m_1,\ldots, m_L \atop w_1,\ldots, w_L}})
= \framebox{1} _{zq} \framebox{1} _{zq^{-1}} 
+\framebox{1} _{zq} \framebox{2} _{zq^{-1}} 
+\framebox{2} _{zq} \framebox{2} _{zq^{-1}}.
\end{align*}

The general rank case $U_q(A^{(1)}_n)$ is quite parallel.
We consider the weight space 
$\bigoplus \,\C |\alpha_1,\ldots, \alpha_L\rangle 
\in V_{m_1} \otimes \cdots \otimes V_{m_L}$
specified by the following condition on the arrays
$\alpha_i=(\alpha_{i,1},\ldots, \alpha_{i,n+1}) \in B_{m_i}$
\begin{align}\label{wcon}
\sum_{i=1}^L \alpha_{i,a} 
&= \delta_{a,n+1}\sum_{i=1}^L m_i+ N_a-N_{a-1}\quad (a \in[1, n+1]),
\end{align}
where 
$0 \le N_1 \le \cdots \le N_n \le \sum_{i=1}^Lm_i$
and $N_0=N_{n+1}=0$.
Introduce the functions\footnote{
Reflecting Remark \ref{re:anti}, we switch to
the dual tableaux with ``hole" entries meaning
$\framebox{$\overline{1}$} = \framebox{2},
\framebox{$\overline{2}$} = \framebox{1}$ for $n=1$.}
\begin{align}
\framebox{$\overline{a}$}_z &= q^{2N_a}
\frac{Q_{a-1}(q^{a-n}z)Q_a(q^{a-3-n}z)}
{Q_{a-1}(q^{a-2-n}z)Q_a(q^{a-1-n}z)}
\prod_{i=1}^L\left(\frac{q^{-m_i+1}w_i-z}{q^{m_i+1}w_i-z}\right)^{\theta(a\le n)}
\quad (a \in [1,n+1]),
\label{taba}\\
Q_a(z) &= \prod_{k=1}^{N_a}(1-zu^{(a)}_k)\;\;(a \in [1,n]),\quad
Q_0(z)=Q_{n+1}(z)=1.\nonumber
\end{align}
The numbers $\{u^{(a)}_j\mid a \in [1,n], j \in [1,N_a]\}$ are solutions to the
Bethe equation:
\begin{align}\label{bae}
-\prod_{i=1}^L\left(\frac{1-q^{-m_i}w_iu^{(n)}_j}
{1-q^{m_i}w_iu^{(n)}_j}\right)^{\!\delta_{a,n}}
= q^{2N_{a+1}-2N_a}
\frac{Q_{a-1}(q^{-1}/u^{(a)}_j)Q_{a}(q^{2}/u^{(a)}_j)Q_{a+1}(q^{-1}/u^{(a)}_j)}
{Q_{a-1}(q/u^{(a)}_j)Q_{a}(q^{-2}/u^{(a)}_j)Q_{a+1}(q/u^{(a)}_j)}.
\end{align}
The eigenvalues
$\Lambda(l,z|{\textstyle {m_1,\ldots, m_L \atop w_1,\ldots, w_L}})$ of 
(\ref{tdef}) on the subspace with the weight (\ref{wcon}) are expressed as the 
sum over the tableaux:
\begin{align}\label{Lam}
\Lambda(l,z|{\textstyle {m_1,\ldots, m_L \atop w_1,\ldots, w_L}})
&= \sum_{n+1\ge a_1 \ge a_2 \ge \cdots \ge a_l \ge 1}
\framebox{$\overline{a}_1$}_{zq^{l-1}}
\framebox{$\overline{a}_2$}_{z q^{l-3}}
\cdots \;\framebox{$\overline{a}_l$}_{z q^{-l+1}},
\end{align}
where the summands stand for products of (\ref{taba}).
They correspond exactly to the semistandard tableaux 
on $n \times l$ rectangle  
provided that  (\ref{taba}) is regarded as the single column 
filled with $\{1,2,\ldots, n+1\}\setminus \{ a\}$.

\begin{example}
For $n=2$ one has
\begin{align}\label{koy}
\Lambda(1,z|{\textstyle {m_1,\ldots, m_L \atop w_1,\ldots, w_L}})
&= \frac{Q_2(qz)}{Q_2(q^{-1}z)}
+ 
\prod_{i=1}^L\left(\frac{q^{-m_i+1}w_i-z}{q^{m_i+1}w_i-z}\right)
\left(
q^{2N_2}\frac{Q_1(z)Q_2(q^{-3}z)}{Q_1(q^{-2}z)Q_2(q^{-1} z)}
+
q^{2N_1}\frac{Q_1(q^{-4}z)}{Q_1(q^{-2}z)}
\right)
\end{align}
and the Bethe equation:
\begin{align*}
-1 &= q^{2N_{2}-2N_1}
\frac{Q_{1}(q^{2}/u^{(1)}_j)Q_{2}(q^{-1}/u^{(1)}_j)}
{Q_{1}(q^{-2}/u^{(1)}_j)Q_{2}(q/u^{(1)}_j)},\\
-\prod_{i=1}^L\frac{1-q^{-m_i}w_iu^{(2)}_j}{1-q^{m_i}w_iu^{(2)}_j}
&= q^{-2N_2}
\frac{Q_{1}(q^{-1}/u^{(2)}_j)Q_{2}(q^{2}/u^{(2)}_j)}
{Q_{1}(q/u^{(2)}_j)Q_{2}(q^{-2}/u^{(2)}_j)}.
\end{align*}
Examples of actual 
eigenvalues and Bethe roots are available in Example \ref{ex:la}.
\end{example}

In general let us separate the sum (\ref{Lam}) into two cases 
according to $a_l=n+1$ or $a_l \le n$.
The former consists of the single term
corresponding to 
$a_1=\cdots = a_l=n+1$, whereas the latter 
always contains 
$\prod_{i=1}^L\frac{q^{-m_i}w_i-q^{-l}z}{q^{m_i}w_i-q^{-l}z}$.
This leads to the decomposition
\begin{align}\label{Lam2}
\Lambda(l,z|{\textstyle {m_1,\ldots, m_L \atop w_1,\ldots, w_L}})
= \frac{Q_n(q^lz)}{Q_n(q^{-l}z)} + 
\prod_{i=1}^L\left(\frac{q^{-m_i}w_i-q^{-l}z}{q^{m_i}w_i-q^{-l}z}\right)
X(z),
\end{align}
where $X(z)$ is a rational function without a pole at 
$z = q^l$ in general.

\subsection{\mathversion{bold}Spectrum of $T(l|m_1,\ldots, m_L)$}
Now we are ready to derive the spectrum of 
the discrete time Markov matrix 
$T(l|m_1,\ldots, m_L)$ in (\ref{spee}).
Under the specialization $z=q^l$ and $w_i= q^{m_i}$,
the second term in (\ref{Lam2}) vanishes, therefore 
the eigenvalue formula takes the factorized form
\begin{align}\label{Lam2.5}
\Lambda(l,q^l|{\textstyle {m_1,\ldots, m_L \atop q^{m_1},\ldots, q^{m_L}}})
&= \frac{Q_n(q^{2l})}{Q_n(1)} = 
\prod_{j=1}^{N_n}\frac{1-q^{2l}u^{(n)}_j}{1-u^{(n)}_j}
\end{align}
in terms of $u^{(n)}_j$'s that are determined from the 
specialized Bethe equation:
\begin{align}\label{bae2}
-\prod_{i=1}^L\left(\frac{1-u^{(n)}_j}{1-q^{2m_i}u^{(n)}_j}\right)^{\!\delta_{a,n}}
= q^{2N_{a+1}-2N_a}
\frac{Q_{a-1}(q^{-1}/u^{(a)}_j)Q_{a}(q^{2}/u^{(a)}_j)Q_{a+1}(q^{-1}/u^{(a)}_j)}
{Q_{a-1}(q/u^{(a)}_j)Q_{a}(q^{-2}/u^{(a)}_j)Q_{a+1}(q/u^{(a)}_j)}.
\end{align}

\subsection{\mathversion{bold}Spectrum of 
$\mathscr{T}(\lambda|\mu_1,\ldots, \mu_L)$}

The Markov transfer matrix 
$\mathscr{T}(\lambda|\mu_1,\ldots, \mu_L)$ was defined in (\ref{ioa}).
Below we write down 
a natural extrapolation of the results in the previous subsection 
in view of the correspondence (\ref{taio}) 
although their rigorous derivation is yet to be supplied. 

The eigenvalues $\Lambda(\lambda|\mu_1,\ldots, \mu_L)$
of $\mathscr{T}(\lambda|\mu_1,\ldots, \mu_L)$ and the 
Bethe equation are given by
\begin{align}
&\Lambda(\lambda|\mu_1,\ldots, \mu_L) 
= \prod_{j=1}^{N_n}
\frac{1-\lambda^{-1}u^{(n)}_j}{1-u^{(n)}_j},
\label{Lam3}\\
&-\prod_{i=1}^L\left(
\frac{1-u^{(n)}_j}{1-\mu^{-1}_iu^{(n)}_j}\right)^{\!\delta_{a,n}}
=
\prod_{k=1}^{N_{a-1}}\frac{u^{(a)}_j-u^{(a-1)}_k}{u^{(a)}_j-qu^{(a-1)}_k}
\prod_{k=1}^{N_{a}}\frac{u^{(a)}_j-qu^{(a)}_k}{qu^{(a)}_j-u^{(a)}_k}
\prod_{k=1}^{N_{a+1}}\frac{qu^{(a)}_j-u^{(a+1)}_k}{u^{(a)}_j-u^{(a+1)}_k},
\label{bae3}
\end{align}
where $q^{1/2}$ has been avoided by replacing 
$u^{(a)}_j$ in $(\ref{bae2})|_{q\rightarrow q^{1/2}}$  
with $q^{(n-a)/2}u^{(a)}_j$. 

\subsection{\mathversion{bold}Spectrum of $\tau(\lambda|\mu)$, $H$ and $\hat{H}$}
Let us further specialize (\ref{Lam3}) and (\ref{bae3}) 
so as to fit 
$\tau(\lambda|\mu)$ in (\ref{tau}).
By setting $\mu_i=\mu$,
the eigenvalues of $\tau(\lambda|\mu)$ (denoted by the same symbol) 
and the relevant Bethe equation are given by 
\begin{align}
&\tau(\lambda|\mu) = \prod_{j=1}^{N_n}
\frac{1-\lambda^{-1}u^{(n)}_j}{1-u^{(n)}_j},
\label{Lam4}\\
&-\left(
\frac{1- u^{(n)}_j}{1-\mu^{-1}u^{(n)}_j}\right)^{\!\!L\delta_{a,n}}
=
\prod_{k=1}^{N_{a-1}}\frac{u^{(a)}_j-u^{(a-1)}_k}{u^{(a)}_j-qu^{(a-1)}_k}
\prod_{k=1}^{N_{a}}\frac{u^{(a)}_j-qu^{(a)}_k}{qu^{(a)}_j-u^{(a)}_k}
\prod_{k=1}^{N_{a+1}}\frac{qu^{(a)}_j-u^{(a+1)}_k}{u^{(a)}_j-u^{(a+1)}_k}.
\label{bae4}
\end{align}
When $n=1$, these results reduce to \cite[eq.(38) and Bethe eq.~on p17]{P}
by replacing $(u^{(1)}_j, \mu,\lambda)$ with $(\nu u_j, \nu, \nu/\mu)$.
From (\ref{bax}), eigenvalues of the continuous time 
Markov matrices $H$ and $\hat{H}$ (denoted by the same symbols) 
are obtained by differentiation with respect to $\lambda$.
Since the Bethe roots are independent of $\lambda$, they are given by
\begin{align}\label{He}
H = -\epsilon \sum_{j=1}^{N_n}\frac{ \mu^{-1}u^{(n)}_j}{1-u^{(n)}_j},
\qquad
\hat{H} = \epsilon \sum_{j=1}^{N_n}\frac{u^{(n)}_j}{\mu-u^{(n)}_j}
\end{align}
in terms of solutions to the same Bethe equation (\ref{bae4}).
One can detect the ``spectral equivalence" implied by (\ref{HH}) 
also from the Bethe ansatz result here.
Denote the system of Bethe equations (\ref{bae4}) symbolically by 
$\mathcal{B}(\{u^{(a)}_j\},q,\mu)$ and the eigenvalue formulas
(\ref{He}) by $H(\{u^{(n)}_j\}, \epsilon, \mu)$ and 
$\hat{H}(\{u^{(n)}_j\}, \epsilon, \mu)$.
Then it is easy to see that 
$\mathcal{B}(\{u^{(a)}_j\},q,\mu)$ is equivalent to 
$\mathcal{B}(\{v^{(a)}_j\},q^{-1},\mu^{-1})$ with $v^{(a)}_j = \mu^{-1}q^{n-a}u^{(a)}_j$
and 
$\mu^{-1}H(\{v^{(n)}_j\}, -\epsilon,\mu^{-1})
= \hat{H}(\{u^{(n)}_j\}, \epsilon, \mu)$.

\subsection{Steady state eigenvalue}

The steady states in the discrete and continuous time
Markov processes are characterized as the one-dimensional subspace having
eigenvalues 1 and 0 for the relevant Markov matrices, respectively.
In our case, they correspond to the solution of the Bethe equation such that
$\forall u^{(n)}_j = 0$ in (\ref{Lam2.5}),  (\ref{Lam3}), (\ref{Lam4}) and (\ref{He}).

For $n=1$, there remains no other Bethe equation to be solved,
indicating that the steady state is uniform 
(or possesses a product measure at most)  
under the periodic boundary condition
as emphasized in \cite{P, EMZ}.
In general the steady state for $n\ge 2$ is nontrivial.
However at least on the level of Bethe roots, they exhibit the same simplifying
feature as the $n=1$ case.
The following example is an exposition of this fact.
\begin{example}\label{ex:la}
Let $n=2$ and 
consider the transfer matrix $T(1,z|{1,1,1 \atop 1,1,1})$ (\ref{tdef}) for the length  
$L=3$ chain.
We concentrate on the sector 
specified by $(N_1,N_2)=(1,2)$ in (\ref{wcon}).
It is the six-dimensional space 
$\oplus_{(i,j,k):\text{permutations of } (1,2,3)} \C|i,j,k\rangle$,
where $1=(1,0,0), 2=(0,1,0), 3=(0,0,1)$ in the previous notation.
The six eigenvalues denoted by 
$\Lambda_1, \Lambda_2, 
\Lambda_3^{\pm},
\Lambda_4^{\pm}$ and 
the corresponding Baxter $Q$ functions 
$Q_1=Q_1(z)$ and $Q_2=Q_2(z)$ by which 
they are expressed as 
$\Lambda(1,z|{1,1,1 \atop 1,1,1}) 
= \Lambda(1,qz|{1,1,1 \atop q,q,q})$ in (\ref{koy}) are given as follows:
\begin{align*}
\Lambda_1 &=
1 + \left(\frac{1-z}{q^2-z}\right)^3(q^2+q^4),
\qquad
Q_1= 1, \quad Q_2=1,\\
\Lambda_2 &= \frac{3 q^6 z^2-3 q^6 z+q^6-q^4 z^3
-3 q^4 z+q^4-q^2 z^3+3 q^2 z^2+q^2-z^3+3 z^2-3z}{\left(q^2-z\right)^3},\\
 &Q_1 = 1-\frac{3 q^2 z}{\left(q^2-q+1\right) \left(q^2+q+1\right)},
 \quad
 Q_2 =\frac{3 q^2 z^2}{\left(q^2-q+1\right) \left(q^2+q+1\right)}-
 \frac{3 q \left(q^2+1\right)z}
 {\left(q^2-q+1\right) \left(q^2+q+1\right)}+1,\\
 \Lambda_3^{\pm} &=
 -\frac{3 q^6 z-2 q^6+2 q^4 z^3-12 q^4 z^2
 +9 q^4 z-2 q^4+2 q^2 z^3+3 q^2 z\pm i \sqrt{3}
\left(q^2-1\right)^3 z-2 q^2+2 z^3-6 z^2+3 z}{2 \left(q^2-z\right)^3},\\
&Q_1=
 1+\frac{\left(\sqrt{3} q^2\mp 3 i q^2-\sqrt{3}\mp 3 i\right) q^2 z}
 {2 (-1\pm i q) (q\mp i)
\left(q^2-q+1\right) \left(q^2+q+1\right)},\quad
Q_2= 1\mp \frac{i q \left(\sqrt{3} q^2 \mp 3 i q^2-\sqrt{3} \mp 3 i\right) z}
{2 \left(q^2-q+1\right)\left(q^2+q+1\right)},\\
\Lambda_4^{\pm} &=
\frac{3 q^6 z^2-6 q^6 z+2 q^6-2 q^4 z^3+3 q^4 z^2
+2 q^4-2 q^2 z^3+9 q^2 z^2\mp i\sqrt{3}
\left(q^2-1\right)^3 z^2-12 q^2 z+2 q^2-2 z^3+3 z^2}{2\left(q^2-z\right)^3},\\
&Q_1=1,\quad 
Q_2 = \frac{q^2 \left(\sqrt{3} q^2\pm 3 i q^2-\sqrt{3} \pm3 i\right) z^2}
{2 (-1\pm i q) (q\mp i)\left(q^2-q+1\right)
\left(q^2+q+1\right)}-\frac{3 q z}{q^2+1}+1.
\end{align*}
Note that $\Lambda_1=1$  under the 
specialization $z=1$ to the stochastic point.
\end{example}

General case is similar.
We conjecture that the unique eigenvalue 
$\Lambda_{\mathrm{sst}}
(l,z|{\textstyle {m_1,\ldots, m_L \atop w_1,\ldots, w_L}})$ 
relevant to the steady state 
corresponds to the Baxter $Q$ functions $\forall Q_a(z)=1$  in (\ref{Lam}),
or equivalently $\forall u^{(a)}_j = 0$.
From (\ref{taba}) and (\ref{Lam}), it reads explicitly as
\begin{align}\label{esst}
\Lambda_{\mathrm{sst}}
(l,z|{\textstyle {m_1,\ldots, m_L \atop w_1,\ldots, w_L}})
=1 + \sum_{r=1}^l \prod_{i=1}^L
\frac{(q^{l-m_i}w_i/z;q^{-2})_{l-r+1}}{(q^{l+m_i}w_i/z;q^{-2})_{l-r+1}}
\sum_{n \ge a_r \ge a_{r+1} \ge \cdots \ge a_l \ge 1}
q^{2N_{a_r}+2N_{a_{r+1}}+\cdots + 2N_{a_l}}.
\end{align}
It indeed satisfies 
$\Lambda_{\mathrm{sst}}
(l,z=q^l|{\textstyle {m_1,\ldots, m_L \atop q^{m_1},\ldots, q^{m_L}}})=1$.
This eigenvalue is exceptional in that 
the Bethe equations are trivially satisfied\footnote{To see 
$\forall u^{(a)}_j = 0$ is a solution of the 
Bethe equation, multiply (\ref{bae3}) or (\ref{bae4})
by their denominators.}.

On the other hand, steady states themselves  
are nontrivial for multispecies case $n\ge 2$.
\begin{example}\label{ex:ss1}
In the $n=2$ species continuous time Markov process (\ref{ctmp2}) with the 
local transition rate (\ref{rate2}) and $\epsilon=1$, 
the (unnormalized) steady state in the 
sector $(N_1,N_2)=(1,2)$ for $L=3,4$ takes the form
$|\bar{P}_L\rangle + \text{cyclic permutations}$ with 
\begin{align*}
|\bar{P}_3\rangle
&=3(1-q\mu)|\emptyset,\emptyset,12\rangle 
+(2+q)(1-\mu)|\emptyset,2,1\rangle
+(1+2q)(1-\mu)|\emptyset,1,2\rangle,\\
|\bar{P}_4\rangle
&=4(1-q\mu)|\emptyset,\emptyset,\emptyset, 12\rangle
+(3+q)(1-\mu)|\emptyset,\emptyset,2,1\rangle
+2(1+q)(1-\mu)|\emptyset,1,\emptyset, 2\rangle
+(1+3q)(1-\mu)|\emptyset,\emptyset, 1,2\rangle,
\end{align*}
where $\emptyset = (0,0), 1=(1,0), 2=(0,1)$ and $12=(1,1)$.

The same data for the model with the adjacent transition rate 
(\ref{qKMO}) read
\begin{align*}
|\bar{P}'_3\rangle
&=3|\emptyset,\emptyset,12\rangle 
+(2+q)|\emptyset,2,1\rangle
+(1+2q)|\emptyset,1,2\rangle,\\
|\bar{P}'_4\rangle
&=4|\emptyset,\emptyset,\emptyset, 12\rangle
+(3+q)|\emptyset,\emptyset,2,1\rangle
+2(1+q)|\emptyset,1,\emptyset, 2\rangle
+(1+3q)|\emptyset,\emptyset, 1,2\rangle,
\end{align*}
which indeed agree with $|\bar{P}_3\rangle$ and $|\bar{P}_4\rangle$ with $\mu=0$.
In another sector $(N_1,N_2)=(2,3)$, the corresponding data are given by
\begin{align*}
|\bar{P}''_3\rangle
&=3(2+q)|\emptyset,\emptyset,112\rangle 
+3(1+q+q^2)|\emptyset,2,11\rangle
+ 3(1+q)(1+q+q^2)|\emptyset,1,12\rangle\\
&+ (1+q)(5+2q+2q^2)|\emptyset,12,1\rangle
+ (1+2q+5q^2+q^3)|\emptyset,11,2\rangle
+ (1+q)(2+q)(1+q+q^2)|1,1,2\rangle,\\
|\bar{P}''_4\rangle
&=2(5+3q)|\emptyset,\emptyset, \emptyset,112\rangle 
+2 (3 + 3 q + 2 q^2)|\emptyset,\emptyset, 2,11\rangle 
+2 (1 + q) (2 + 3 q + 3 q^2)  |\emptyset,\emptyset, 1,12\rangle \\
&+(1 + q) (9 + 4 q + 3 q^2)|\emptyset,\emptyset, 12,1\rangle 
+(1+3q+9q^2+3q^3)|\emptyset,\emptyset, 11,2\rangle
+(3 + 5 q + 7 q^2 + q^3) |\emptyset, 2,\emptyset, 11\rangle\\
&+(1 + q) (5 + 5 q + 5 q^2 + q^3)|\emptyset, 2,1,1\rangle
+(1 + q) (7 + 4 q + 5 q^2) |\emptyset, 1,\emptyset, 12\rangle\\
&+(1 + q)^2(3 + 3 q + 2 q^2) |\emptyset, 1,2,1\rangle
+(1 + q)^2 (2 + 3 q + 3 q^2) |\emptyset, 1,1,2\rangle,
\end{align*}
where $11=(2,0)$ and $112=(2,1)$.
The specialization of $|\bar{P}'_3\rangle$, $|\bar{P}'_4\rangle$ and $|\bar{P}''_3\rangle$ 
at $q=0$  exactly reproduce 
$|\xi_3(1,1)\rangle$, $|\xi_4(1,1)\rangle$ and $|\xi_3(2,1)\rangle$
available in \cite[Ex. 2.1$|_{\forall  w_a=1}$]{KMO}. 
It is notable that the (unnormalized) steady state probabilities are polynomials in $q$ with
{\em nonnegative} integer coefficients.
\end{example}

As these examples indicate, 
steady states for multispecies case $n \ge 2$ are involved but 
{\em algebraic}\footnote{Of course this must be so 
since the null space of the Markov matrix is one-dimensional and their
elements are algebraic.}  in that 
no transcendental input from nontrivial solutions to the Bethe equation 
is required.
The steady states are known to 
exhibit rich combinatorial and algebraic structures related to the 
crystal base of quantum groups and the tetrahedron equation 
already at $q=0$ \cite{KMO2}.
Their systematic investigation will be presented elsewhere.

\section{Summary}\label{sec:end}

In this paper we have explored new prospects of  
the $U_q(A^{(1)}_n)$ quantum $R$ matrix 
for the symmetric tensor representation $V_l \otimes V_m$
which have applications to integrable stochastic models in 
non-equilibrium statistical mechanics.

The $R$ matrix $R(z)$ has been shown to factorize at $z=q^{l-m}$ 
for $l \le m$ from which the non-negativity is manifest 
in an appropriate range of the remaining parameters (Theorem \ref{pr:Rs}).
We have found a suitable gauge $S(z)$ (\ref{SR}) of $R(z)$ 
which satisfies  the sum rule (Theorem \ref{pr:s=1}) as well as the 
Yang-Baxter equation (Proposition \ref{pr:S}).
We have also introduced the specialized $S$ matrix 
$\mathscr{S}(\lambda, \mu)$ corresponding to the 
extrapolation of $S(z=q^{l-m})$ to generic $l, m$.
It also satisfies the non-negativity, the sum rule (\ref{psum}), (\ref{syk}) and 
the Yang-Baxter equation without ``difference property" (Remark \ref{re:dp}).
 
Based on the stochastic $R$ matrices $S(z)$ and $\mathscr{S}(\lambda, \mu)$,
we have constructed new integrable Markov chains described in terms of
$n$ species of particles obeying asymmetric dynamics.
They are discrete time systems with (Section \ref{ss:dmc1}) 
and without (Section \ref{ss:dmc2})  
constraints on the number of particles at lattice sites 
and those hopping to the neighboring site at one time step.
The other ones (Section \ref{ss:ctm}) are $n$-species TAZRPs
corresponding to continuous time limits of that in Section \ref{ss:dmc2}.
Two such TAZRPs associated to the ``Hamiltonian points" $\lambda=1$ and $\lambda=\mu$ 
of the Markov transfer matrix are obtained and their interrelation  (\ref{HH}) has been clarified.
They admit a superposition yielding an integrable asymmetric zero range process
in which $n$ species of particles can hop to either direction (Remark \ref{re:H}).

The Markov matrices in these models are specializations 
of the commuting transfer matrices whose spectra are well-known 
by the Bethe ansatz in the theory of quantum integrable systems.
However, the precise adjustment to the present stochastic setting
demands some work. 
We have given the resulting Bethe eigenvalue formulas for all the 
models under the periodic boundary condition (Section \ref{sec:ba}).
In particular, the eigenvalues relevant to the steady states are
found to correspond to the trivial choice $\forall Q_a(z)=1$ of the  
Baxter $Q$ functions.
This explains the algebraic (non-transcendental)
nature of the steady states from the Bethe ansatz point of view,
indicating a possible 
alternative approach by the method of {\em matrix products}. 
These issues will be addressed elsewhere.

\appendix

\section{Example of explicit forms of quantum $R$ matrices}\label{app:R}

For $U_q(A^{(1)}_n)$, 
the matrix elements of $R(z)$ on $V_1 \ot V_m$ are as follows:
\begin{align*}
R(z)^{e_k,\delta}_{e_j,\beta}=
\begin{cases}
q^{\beta_k+1}\frac{1-q^{-2\delta_k+m-1}z}{q^{m+1}-z} & \text{if }\;j=k\\
-q^{\beta_{j+1}+\cdots+\beta_{k-1}}\frac{1-q^{2\beta_k}}{q^{m+1}-z} & \text{if } \;j<k,\\
-q^{m-(\beta_{k}+\cdots+\beta_{j})}
\frac{z(1-q^{2\beta_k})}{q^{m+1}-z} & \text{if } \; j>k,
\end{cases}
\end{align*}
where $e_j=(0,\cdots,0,1,0,\cdots,0) \in \Z^{n+1}$ contains 
$1$ at the $j$-th position from the left.
Similarly the matrix elements of $R(z)$ on $V_l \ot V_1$ are as follows:
\begin{align*}
R(z)^{\gamma,e_k}_{\alpha,e_j}=
\begin{cases}
q^{\alpha_k+1}\frac{1-q^{-2\alpha_k+l-1}z}{q^{l+1}-z} & \text{if }\;j=k\\
-q^{l-(\alpha_j+\cdots + \alpha_k)}
\frac{z(1-q^{2\alpha_k})}{q^{l+1}-z} & \text{if } \;j<k,\\
-q^{\alpha_{k+1}+\cdots+\alpha_{j-1}}
\frac{1-q^{2\alpha_k}}{q^{l+1}-z} & \text{if } \; j>k.
\end{cases}
\end{align*}
For $U_q(A^{(1)}_1)$, the $R$ matrix on $V_2 \ot V_2$ 
defines a 19-vertex model. 
Its action is given by 
\begin{align*}
&R(z)(|02\rangle \ot |02\rangle)=|02\rangle \ot |02\rangle, 
\quad R(z)(|20\rangle \ot |20\rangle)=|20\rangle \ot |20\rangle,\\
&R(z)(|02\rangle \ot |11\rangle)=\frac{q^2(1-z)}{q^4-z}|02\rangle \ot |11\rangle-\frac{(1-q^4)z}{q^4-z}|11\rangle \ot |02\rangle, \\
&R(z)(|02\rangle \ot |20\rangle)=\frac{q^2(1-z)(1-q^2z)}{(q^2-z)(q^4-z)}|02\rangle \ot |20\rangle-\frac{q(1+q^2)(1-q^4)(1-z)z}{(q^2-z)(q^4-z)}|11\rangle \ot |11\rangle \\
& \ \ \ \ \ \ \ \ \ \ \ \ \ \ \ \ \ \ \ \ \ \ \ \ \   +\frac{(1-q^2)(1-q^4)z^2}{(q^2-z)(q^4-z)}|20\rangle \ot |02\rangle, \\
&R(z)(|11\rangle \ot |20\rangle)=-\frac{(1-q^4)z}{q^4-z}|20\rangle \ot |11\rangle+\frac{q^2(1-z)}{q^4-z}|11\rangle \ot |20\rangle, \\
&R(z)(|11\rangle \ot |11\rangle)= -\frac{q(1-q^2)(1-z)}{(q^2-z)(q^4-z)}|02\rangle \ot |20\rangle+\frac{q^6z-2q^4z+q^4+q^2z^2-2q^2z+z}{(q^2-z)(q^4-z)}|11\rangle \ot |11\rangle \\
& \ \ \ \ \ \ \ \ \ \ \ \ \ \ \ \ \ \ \ \ \ \ \ \ \   -\frac{q(1-q^2)(1-z)z}{(q^2-z)(q^4-z)}|20\rangle \ot |02\rangle, \\
&R(z)(|11\rangle \ot |02\rangle)=\frac{q^2(1-z)}{q^4-z}|11\rangle \ot |02\rangle-\frac{1-q^4}{q^4-z}|02\rangle \ot |11\rangle, \\
&R(z)(|20\rangle \ot |02\rangle)=\frac{q^2(1-z)(1-q^2z)}{(q^2-z)(q^4-z)}|20\rangle \ot |02\rangle-\frac{q(1+q^2)(1-q^4)(1-z)}{(q^2-z)(q^4-z)}|11\rangle \ot |11\rangle \\
& \ \ \ \ \ \ \ \ \ \ \ \ \ \ \ \ \ \ \ \ \ \ \ \ \  +\frac{(1-q^2)(1-q^4)}{(q^2-z)(q^4-z)}|02\rangle \ot |20\rangle, \\
&R(z)(|20\rangle \ot |11\rangle)=
\frac{q^2(1-z)}{q^4-z}|20\rangle \ot |11\rangle-\frac{1-q^4}{q^4-z}|11\rangle \ot |20\rangle,
\end{align*}
where $|\alpha\rangle$ with $\alpha=(\alpha_1,\alpha_2)$ is 
denoted by $|\alpha_1\alpha_2\rangle$.

Similarly the $U_q(A^{(1)}_2)$ $R$ matrix on $V_2 \ot V_2$ defines a 
102-vertex model. We present some examples of its action.
\begin{align*}
&R(z)(|u\rangle \ot |u\rangle) = |u\rangle \ot |u\rangle,\;\;
\text{for }\; | u \rangle = |002\rangle, |020\rangle, |200\rangle,\\
&R(z)(|002\rangle \ot |011\rangle)=\frac{q^2(1-z)}{q^4-z}|002\rangle \ot |011\rangle-\frac{(1-q^4)z}{q^4-z}|011\rangle \ot |002\rangle, \\
&R(z)(|002\rangle \ot |020\rangle)=\frac{q^2(1-z)(1-q^2z)}{(q^2-z)(q^4-z)}|002\rangle \ot |020\rangle-\frac{q(1+q^2)(1-q^4)(1-z)z}{(q^2-z)(q^4-z)}|011\rangle \ot |011\rangle \\
& \ \ \ \ \ \ \ \ \ \ \ \ \ \ \ \ \ \ \ \ \ \ \ \ \  \ \   +\frac{(1-q^2)(1-q^4)z^2}{(q^2-z)(q^4-z)}|020\rangle \ot |002\rangle, \\
&R(z)(|002\rangle \ot |110\rangle)=\frac{q^2(1-z)(1-q^2z)}{(q^2-z)(q^4-z)}|002\rangle \ot |110\rangle-\frac{q^2(1-q^4)(1-z)z}{(q^2-z)(q^4-z)}|011\rangle \ot |101\rangle \\
& \ \ \ \ \ \ \ \ \ \ \ \ \ \ \ \ \ \ \ \ \ \ \ \ \  \ \   -\frac{q(1-q^4)(1-z)z}{(q^2-z)(q^4-z)}|101\rangle \ot |011\rangle+\frac{(1-q^2)(1-q^4)z^2}{(q^2-z)(q^4-z)}|110\rangle \ot |002\rangle, \\
&R(z)(|011\rangle \ot |011\rangle)=-\frac{q(1-q^2)(1-z)}{(q^2-z)(q^4-z)}|002\rangle \ot |020\rangle+\frac{q^4+z-2q^2z-2q^4z+q^6z+q^2z^2}{(q^2-z)(q^4-z)}|011\rangle \ot |011\rangle \\
& \ \ \ \ \ \ \ \ \ \ \ \ \ \ \ \ \ \ \ \ \ \ \ \ \  \ \    -\frac{q(1-q^2)(1-z)z}{(q^2-z)(q^4-z)}|020\rangle \ot |002\rangle. 
\end{align*}

The $U_q(A^{(1)}_2)$ $R$ matrix on $V_2 \ot V_3$ defines a 
204-vertex model. 
Let us pick the three matrix elements $R(z)_{\alpha, \beta}^{\gamma,\delta}$ 
having the common $(\beta, \gamma)=(201,101)$ as
\begin{align*} 
R(z)_{002,201}^{101,102} &= -\frac{q(1+q^2)(1-q^4)(q-z)z}{(q^3-z)(q^5-z)},\quad
R(z)_{011,201}^{101,111} =\frac{z(1-q^4)(1-q^2-q^4+q^3z)}{(q^3-z)(q^5-z)},\\
R(z)_{020,201}^{101,120} &= \frac{(1-q^4)^2z}{(q^3-z)(q^5-z)}.
\end{align*}
Note that $\beta \ge \gamma$ is satisfied.
For comparison we also consider the two elements with 
$(\beta, \gamma)=(201,110)$ breaking $\beta \ge \gamma$:
\begin{align*}
R(z)_{020,201}^{110,111} &= -\frac{q^2(1+q^2)(1-q^4)z(1-qz)}{(q^3-z)(q^5-z)},
\quad
R(z)_{110,201}^{110,201} =\frac{q^3(q-z)(1-qz)}{(q^3-z)(q^5-z)}.
\end{align*}
In the former three,  
$\psi^{\gamma,\delta}_{\alpha,\beta} = 1$ holds in (\ref{psi}), thus we find 
\begin{align*}
R(q^{-1})_{002,201}^{101,102} 
= R(q^{-1})_{011,201}^{101,111}
= R(q^{-1})_{020,201}^{101,120}  = \frac{q(1-q^4)}{1-q^6}  
= q \binom{3}{2}_{\!q^2}^{-1} \binom{2}{1}_{\!q^2}
\end{align*}
and 
$R(q^{-1})_{020,201}^{110,111}  = R(q^{-1})_{110,201}^{110,201} = 0$
in agreement with Theorem \ref{pr:Rs}.

\section*{Acknowledgments}
The authors thank Yoshihiro Takeyama for communication on references.
A.K. thanks Rodney Baxter, Vladimir Bazhanov and Sergey Sergeev for warm 
hospitality at Australian National University where a part of this work was done.
This work is supported by 
Grants-in-Aid for Scientific Research No.~15K04892,
No.~15K13429 and No.~23340007 from JSPS.

\end{document}